\documentclass[12pt, a4paper]{amsart}
\usepackage{amssymb}
\usepackage{amscd}
\usepackage{verbatim}
\usepackage[curve,matrix,arrow]{xy}
\usepackage{enumerate}
\theoremstyle{plain}\newtheorem{Theorem}{Theorem}[section]
\theoremstyle{plain}
\theoremstyle{plain}\newtheorem{Corollary}[Theorem]{Corollary}
\theoremstyle{plain}\newtheorem{Lemma}[Theorem]{Lemma}
\theoremstyle{plain}\newtheorem{Proposition}[Theorem]{Proposition}
\theoremstyle{definition}\newtheorem{Definition}[Theorem]{Definition}
\theoremstyle{definition}\newtheorem{Example}[Theorem]{Example}
\theoremstyle{definition}
\theoremstyle{definition}\newtheorem{Remark}[Theorem]{Remark}
\newcommand{\ilim}{\mathop{\varprojlim}\limits}

\DeclareMathOperator{\ess}{ess}
\DeclareMathOperator{\clo}{clo}
\DeclareMathOperator{\Gl}{Gl}

\newcommand*\oline[1]{%
  \,\vbox{%
    \hrule height 0.5pt
    \kern0.25ex
    \hbox{%
      \kern-0.2em
      \ifmmode#1\else\ensuremath{#1}\fi
      \kern-0.05em
    }%
  }%
}

\def\N#1#2{N_{#1}(#2)}

\def\osc{{\rm osc}}

\def\GL{{\rm GL}}

\def\Aut{{\rm Aut}}

\def\Ker{{\rm Ker}}

\def\Iso{{\rm Iso}}
\def\Hom{{\rm Hom}}
\def\Mor{{\rm Mor}}
\def\Alp{{\rm Alp}}

\def\id{{\rm id}}
\def\ps{{\rm pro-sat}}
\def\inc{{\rm inc}}
\def\conj{{\rm conj}}

\def\F{{\bf F}}

\def\Ss{{\mathcal S}}
\def\Ff{{\mathcal F}}
\def\Ee{{\mathcal E}}
\def\Gg{{\mathcal G}}

\newcommand{\npar}{\smallskip\par\noindent\pagebreak[2]\refstepcounter{Theorem}{\bf \thesection.\arabic{Theorem}.\ \ }}

\begin{document}

\title{Fusion Systems for Profinite Groups}
\author{Radu Stancu}
\address{LAMFA-CNRS UMR 7352\\
Univ. de Picardie\\
33, Rue Saint-Leu\\
80039 Amiens CX 1\\
France}
\email{radu.stancu@u-picardie.fr}

\author{Peter Symonds}
\address{School of Mathematics\\
         University of Manchester\\
     Manchester M13 9PL\\
     United Kingdom}
\email{Peter.Symonds@manchester.ac.uk}
\date{\today}
\begin{abstract} We introduce the notion of a pro-fusion system on a pro-$p$ group, which generalizes the notion of a fusion system on a finite $p$-group. We also prove a version of Alperin's Fusion Theorem for pro-fusion systems.
\end{abstract}
\maketitle

\section{Introduction}

Profinite groups have a good theory of Sylow pro-$p$ subgroups, so they form an obvious candidate for a generalization of fusion theory in finite groups. See for example the work of Gilotti, Ribes and Serena~\cite{gilottiribesserena} on fusion and transfer in the context of profinite groups.

In the present work we develop a theory of fusion systems for profinite groups. We define a pro-fusion system $\Ff$ on a pro-$p$ group $S$ as an inverse limit of fusion systems on finite $p$-groupe (see Definition~\ref{ConstrProFusSt}) and we study morphisms between them and also their quotients. When the fusion systems in the inverse limit are all saturated we call $\Ff$ a pro-saturated fusion system.

If $\Ff$ is a pro-saturated fusion system on a pro-$p$ group $S$, then analogously to the case of profinite groups, $\Ff$ is an inverse limit of its quotients by open, strongly closed subgroups of $S$ (Proposition~\ref{Proposition7}). We also define saturation of $\Ff$ in terms of  properties familiar from the finite case and show that these are automatically satisfied when $\Ff$ is pro-saturated and $S$ is countably based (Theorem~\ref{prosatsat}).

The main results of this work are concerned with a version of Alperin's Fusion Theorem for pro-saturated fusion systems. We generalize the notion of an invariant subsystem on a strongly closed subgroup $T$ of $S$, by extending the sets of objects to all subgroups of $S$; we call it a $T$-subsystem (Definition~\ref{subsystem}). A saturated $T$-subsystem of a pro-fusion system attempts to mimic a normal subgroup of a profinite group. We show that a morphism in a saturated $T$-subsystem $\Ee$ between two open subgroups of $S$ can be expressed as a composition of a chain of restrictions of automorphisms in $\Ee$ (Theorem~\ref{openAFT}). This relative form of Alperin's Theorem appears to be new even in the case of finite groups. The result for a morphism between closed subgroups is more complicated and we only prove it with additional hypotheses (Theorem~\ref{ProAFTgene}). Moreover, the resulting chain of morphisms need not be finite, although it converges in the profinite topology. In the last section we give an example of a profinite group and a morphism in its associated pro-fusion system for which the length of such an Alperin  chain must be infinite.

\section{Terminology and basic properties}
\npar
When we deal with profinite groups we always work in the category of profinite groups, so all subgroups are closed and all homomorphisms are continuous. For more information on profinite groups see, for example, \cite{rz, serre, wilson}

There is a good theory of Sylow pro-$p$ subgroups. However, there are some differences from the case of finite $p$-groups. One is that a proper subgroup of a pro-$p$ group can be equal to its normalizer; the normalizer only has to be larger when the subgroup is open. Another difference is that a group can be isomorphic to one of its proper subgroups.

One result that does carry over, and which we shall use frequently, is the following: if $P$ is a pro-$p$ group, $Q$ an open subgroup and $N$ a normal subgroup such that $N_N(Q) \leq Q$, then $N \leq Q$.

\npar\label{FinFusSt} Let $p$ be a prime and $S$ a finite $p$-group. A fusion system on $S$ is a category $\Ff$ with objects the set of subgroups of $S$ and morphisms which are injective group homomorphisms subject to two conditions:
\begin{enumerate}[(a)]
\item $\Ff$ contains every group homomorphism induced by conjugation by an element of $S$, and
\item any morphism in $\Ff$ factors as an isomorphism in $\Ff$ followed by an inclusion.
\end{enumerate}
A fusion system $\Ff$ is called saturated if it satisfies a certain extra condition based on Sylow theory (see Definition~\ref{defSaturated} below). For more information on fusion systems see, for example, \cite{BLO2, cravenbook, puig}.

\begin{Definition}
\label{def:mor}
Let $\Ff$ be a fusion system on $S$ and $\Gg$ be a fusion system on $T$. A {\it morphism of fusion systems} between $\Ff$ and $\Gg$ is a group homomorphism $\alpha: S \rightarrow T$ such that there exists a functor $A: \Ff \rightarrow \Gg$ with the property that:
\begin{enumerate}[(a)]
\item $A(P)=\alpha(P)$, $\forall P\le S$;
\item for all $\varphi\in\Hom_\Ff(P,Q)$, $\alpha \varphi=A(\varphi)\alpha$.
\end{enumerate}
\end{Definition}

Notice that $A$ is uniquely determined by $\alpha$, provided that it exists. Thus a morphism of fusion systems is just a group homomorphism between the underlying $p$-groups that satisfies certain conditions, to wit that for each $\varphi \in \Hom _\Ff(P,G)$ we have $\varphi(P\cap N) \leq Q \cap N$ and the induced group homomorphism $\oline{\varphi}:P/P \cap N \rightarrow Q/Q \cap N$ is, in fact, a morphism of $\Gg$. We will refer to a morphism as $(\alpha,A)$, $\alpha$ or $A$, as convenient.

\npar\label{lemma4} Morphisms of fusion systems can be composed in the obvious way. It is easy to see that $\Ker(\alpha)$ is a strongly $\Ff$-closed subgroup of $\Ff$. Recall that $Q$ is a {\it strongly $\Ff$-closed subgroup of $S$} if, for all $\varphi\in\Hom_\Ff(R,S)$, $R\le Q$ implies $\varphi(R)\le Q$. Let $A(\Ff)$ be the collection of objects and morphisms of $\Gg$ that are in the image under $\alpha$ of objects and morphisms in $\Ff$. In general, $A(\Ff)$ need not be a fusion system, because it may not be possible to compose morphisms. We can remedy this by defining $\langle A(\Ff) \rangle$ to be the fusion system on $\alpha(S)$  generated by $A(\Ff)$. There is then a morphism of fusion systems $\Ff \rightarrow \langle A(\Ff) \rangle$.

When $Q$ is a strongly $\Ff$-closed subgroup of $S$, $\Ff/Q$ is defined to be the fusion system on $S/Q$ with morphisms all group homomorphisms $\varphi$ satisfying $\varphi\in\Hom_{\Ff/Q}(PQ/Q,S/Q)$ if and only if there exist $\tilde\varphi\in\Hom_\Ff(PQ,S)$ such that $\varphi(uQ)=\tilde\varphi(u)Q$ for all $u\in P$. However, there might not be a morphism from $\Ff$ to $\Ff/Q$. Define $\oline\Ff_Q$ to be the set of homomorphisms of subgroups of $S/Q$ induced by morphisms in $\Ff$, and let $\langle\, \oline\Ff_Q \rangle$ be the fusion system that it generates. There are then clearly morphisms $\Ff \to \langle\oline\Ff_Q \rangle$ and $\Ff/Q \to \langle\oline\Ff_Q \rangle$. If the latter is an isomorphism, there is a morphism $\Ff \to \Ff/Q$ and we call $\Ff/Q$ a quotient of $\Ff$. In the saturated case all of this is well behaved; the following results are due to Puig.

\begin{Theorem}[\cite{puig}]\label{thmPuig0}
Let $\Ff$ be a saturated fusion system on a finite $p$-group $S$ and let $N \leq S$ be strongly closed. For any $\varphi \in \Hom_\Ff(P,Q)$ there exists $\tilde\varphi \in \Hom_\Ff(PN,QN)$ such that $\tilde\varphi$ and $\varphi$ induce the same homomorphism $PN/N \to QN/N$.
\end{Theorem}

\begin{Corollary}[\cite{puig}]\label{thmPuig}
Let $\Ff$ be a saturated fusion system on $S$.
\begin{enumerate}[(a)]
\item
Let $(\alpha,A):\Ff \rightarrow \Gg$ be a morphism of fusion systems. Then $A(\Ff)$ is a saturated fusion system isomorphic to $\Ff/\Ker(\alpha)$ and $(\alpha,A)$ factors as a quotient followed by an inclusion.
\item
If $Q$ is a strongly $\Ff$-closed subgroup of $S$ then $\Ff/Q$ is a quotient of $\Ff$ (i.e.\ the natural map from $S$ to $S/Q$ induces a morphism of fusion systems). Moreover, $\Ff/Q$ is saturated.
\end{enumerate}
\end{Corollary}

There is a careful discussion of these topics in~\cite{craven}.

We will construct a pro-fusion system by an inverse limit process on a directed system of fusion systems on finite $p$-groups. Recall that a partially ordered set $I$ is {\it directed} if for all $i,j\in I$ there exists $k\in I$, with $k\ge i$ and $k\ge j$.

\begin{Definition}\label{ConstrProFusSt}
Suppose that we have an inverse system of fusion systems $\Ff_i$ on finite $p$-groups $S_i$, indexed by a directed set~$I$. That is, for every $i,j \in I$, $j \geq i$ we have morphisms $(f_{i,j},F_{i,j}):\Ff_j \rightarrow \Ff_i$ such that $f_{i,j}f_{j,k}=f_{i,k}$ whenever $k \geq j \geq i$. Set $S:=\ilim_{i\in I} S_i$ and let $f_i:S\to S_i$ be the induced projections and $N_i:=\Ker(f_i)$. Then $S$ is a pro-$p$ group, and the $\{N_i\mid i\in I\}$ form an open sub-basis for $S$ at $1$. Define $\Ff:=\ilim_{i\in I} \Ff_i$ to be the category on $S$ with objects all the closed subgroups of $S$ and with morphisms given by $\Hom_\Ff(P,Q):=\ilim_{i\in I} \Hom_{\Ff_i}(f_i(P),f_i(Q))$, for all $P,Q\le S$. We say that $\Ff$ is a {\it pro-fusion system} on the pro-$p$ group $S$.
\end{Definition}

Notice that $\Aut_{\Ff}(P)$ is naturally a profinite group. There is a canonical functor $F_i:\Ff\to\Ff_i$ sending $P\le S$ to $f_i(P)\le S_i$ and $\varphi\in\Hom_\Ff(P,Q)$ to its image in $\Hom_\Ff(f_i(P),f_i(Q))$. 

\begin{Example}
If $G$ is a profinite group with Sylow pro-$p$ subgroup $S$, then $\Ff_S(G)$, the category with the closed subgroups of $S$ as objects and as morphisms all homomorphisms induced by conjugation by the elements of $G$, is  a pro-fusion system. For if $G = \ilim_{i\in I}G/N_i$ then $\Ff_S(G)= \ilim_{i\in I} \Ff_{S/S \cap N_i}(G/N_i)$.
\end{Example}

\begin{Lemma}\label{IncIsoDec}
Every morphism in a pro-fusion system $\Ff$ factorizes uniquely as an isomorphism followed by an inclusion, and every endomorphism is an isomorphism.
\end{Lemma}

\begin{proof}
Given $\varphi \in \Hom_{\Ff}(P,Q)$, each $F_i(\varphi)$ factorizes as $\inc_i \theta_i$, where $\theta_i$ is an isomorphism and $\inc _i$ is an inclusion. Thus $\varphi$ factorizes as $\ilim_{i\in I} \inc _i \, \ilim_{i\in I} \theta_i$.
The inverse to $\ilim_{i\in I} \theta_i$ is $\ilim_{i\in I} \theta_i^{-1}$, and $\ilim_{i\in I} \inc _i$ is the inclusion of $\varphi (P)$ in $Q$.

If $\varphi$ is an endomorphism then so is $F_i(\varphi)$; thus $\inc_i$ is the identity map, and so is $\ilim_{i\in I} \inc _i$.
\end{proof}

Sometimes we will write $\varphi ^{-1}$ for the inverse of the isomorphism part of this factorization of $\varphi$.

\begin{Lemma}\label{EI}
Let $\Ff$ be a pro-fusion system on $S$ and $P$, $Q$ subgroups of $S$. Suppose that there are morphisms $\varphi\in\Hom_\Ff(P,Q)$ and $\psi\in\Hom_\Ff(Q,P)$. Then $\varphi$ and $\psi$ are isomorphisms.
\end{Lemma}

\begin{proof}
By Lemma~\ref{IncIsoDec}, $\varphi \psi$ and $\psi \varphi$ are both isomorphisms. Thus both $\varphi$ and $\psi$ are surjective, hence isomorphisms, by Lemma~\ref{IncIsoDec} again.
\end{proof}

\begin{Lemma}
\label{la:open}
If $P$ is open in $S$ and $Q$ is isomorphic to $P$ in $\Ff$, then $Q$ is also open and has the same index in $S$ as $P$.
\end{Lemma}

\begin{proof}
Since $P$ is open in $S$, it must contain some $N_i$, and $Q$ must also contain $N_i$, because $N_i$ is strongly closed. Thus $P/N_i$ is isomorphic to $Q/N_i$ in $\Ff_i$.
\end{proof}

\begin{Lemma}
If $\Ff$ is a pro-fusion system on $S$ then it contains every group homomorphism induced by an element of $S$.
\end{Lemma}

\begin{proof}
Suppose that $\Ff = \ilim \Ff_i$ and that we have $P,Q \leq S$ and $s \in S$ such that ${}^sP \leq Q$. The conjugation homomorphism $c_s: P \rightarrow Q$ has image $c_{f_i(s)} : f_i(P) \rightarrow f_i(Q)$ between subgroups of $S_i$. Since $\Ff_i$ is a fusion system, we know that $c_{f_i(s)} \in \Hom _{\Ff_i}(f_i(P),f_i(Q))$. But $F_{i,j}(c_{f_j(s)})=c_{f_i(s)}$, so we obtain an element of $\Hom _\Ff(P,Q)$ that is clearly equal to $c_s$.
\end{proof}

Thus a pro-fusion satisfies the conditions for a fusion system in \ref{FinFusSt}. However, these are certainly not enough to ensure a good theory.

\begin{Example}
Any attempt to define a pro-fusion system on an infinite pro-$p$ group $S$ by setting $\Hom(P,Q)$ to be all the (continuous) injective group homomorphisms from $P$ to $Q$ will not succeed; in other words, the resulting category cannot be expressed as an inverse limit. For if $S$ is not all torsion then it contains a copy of the $p$-adic integers, which is isomorphic to a proper subgroup of itself, contradicting Lemma~\ref{EI}; if it is torsion then it contains an infinite product of cyclic groups of order $p$, by~\cite{zel}, which has the same property.
\end{Example}

We now define the saturation of a pro-fusion system in the same way as was done for a fusion system on a finite $p$-group by Roberts and Shpectorov \cite{RobertsShpectorov}. We start with some preliminary notions. For the rest of this section $\Ff$ is a pro-fusion system on a pro-$p$ group $S$.

\begin{Definition}
We say that $Q$ is {\it receptive in $\Ff$\/} if for all $Q\le S$ and for all $\varphi\in\Iso_\Ff(R,Q)$, there exists $\tilde \varphi:N_\varphi\to N_S(Q)$ in $\Ff$ such that $\tilde \varphi |_R = \varphi$, where $N_\varphi=\{x\in N_S(R) \mid \exists y\in N_S(Q), \varphi(xux^{-1})=y\varphi(u)y^{-1}, \forall u\in R\}$.
\end{Definition}

Write $\Aut _S(Q)$ for the image of the natural homomorphism $N_S(Q) \rightarrow \Aut _{\Ff}(Q)$. If $K \leq \Aut _{\Ff} (Q)$, write $\Aut_S^K(Q)$ for $\Aut_S(Q)\cap K$ and $N_S^K(Q)$ for the inverse image of $K$ in $N_S(Q)$.

\begin{Definition}
Let $K\le\Aut_\Ff(Q)$. We say that $Q$ is {\it fully $K$-automized in $\Ff$} if
$\Aut_S^K(Q)$ is a Sylow pro-$p$ subgroup of $K$. We say that $Q$ is {\it fully $K$-normalized in~$\Ff$\/} if $Q$ is receptive and fully $K$-automized in $\Ff$.
\end{Definition}

When $K=\Aut_\Ff(Q)$ we say that $Q$ is fully normalized in~$\Ff$ rather than fully $\Aut_\Ff(Q)$-normalized in~$\Ff$. When $K=1$ we say that $Q$ is fully centralized in~$\Ff$.

\begin{Definition}\label{defSaturated}
A pro-fusion system $\Ff$ is {\it saturated} if
 every $\Ff$-isomorphism class contains a subgroup that is fully normalized in~$\Ff$.
\end{Definition}

\begin{Remark}
We follow the treatment of Roberts and Shpectorov, because it does not require us to consider the order of $N_S(Q)/Q$ or of $C_S(Q)Q/Q$, which might be infinite. All the different characterizations in the literature are known to be equivalent in the finite case.
\end{Remark}

\begin{Example}
If $G$ is a profinite group with Sylow pro-$p$ group $S$ then $\Ff_S(G)$ is saturated. This is a consequence of Sylow theory, and the proof is the same as in the finite case.
\end{Example}

The development of the theory from the axioms follows the same lines as in the finite case. We just have to be careful to avoid mentioning the order of a group and not to use any hidden lemma particular to finite groups.

\begin{Lemma}\label{Lemma11}
Let $\Ff$ be a pro-fusion system on $S$, $Q \leq S$ and $K \leq \Aut_{\Ff}(Q)$. If $Q$ is fully $K$-automized in~$\Ff$ and $L\le K\le \Aut_\Ff(Q)$, then there exists a $\kappa \in K$ such that $Q$ is fully $^\kappa L$-automized in~$\Ff$.
\end{Lemma}

\begin{proof}
By hypothesis, $\Aut^K_S(Q)$ is a Sylow pro-$p$ subgroup of $K$, so there is a $\kappa \in K$ such that $\Aut^K_S(Q)$ contains a Sylow pro-$p$ subgroup of $^\kappa L$. Hence $\Aut^{^\kappa L}_S(Q)$ is a Sylow pro-$p$-subgroup of $^\kappa L$.
\end{proof}

\begin{Lemma}
\label{la:equiv_fnorm}

If $Q$ is fully $K$-normalized in $\Ff$ and $\varphi \in \Iso _{\Ff}(R,Q)$, then there is morphism $\psi : N_S^{{}^{\varphi^{-1}}K}(R)R \rightarrow N_S^K(Q)Q$ in $\Ff$ and a $\chi \in K$ such that $\psi |_R = \chi \varphi$.
\end{Lemma}

\begin{proof}
 Consider ${}^\varphi \Aut _S^{{}^{\varphi^{-1}}K}(R) \leq \Aut _{\Ff}^K(Q)$. Since $Q$ is fully $K$-automized, there is a $\chi \in \Aut _{\Ff}^K(Q)$ such that ${}^{\chi \varphi} \Aut _S^{{}^{\varphi^{-1}}K}(R) \leq \Aut _S^K(Q)$. But ${}^{\varphi^{-1}}K={}^{(\chi \varphi)^{-1}}K$, so $N_S^{{}^{\varphi^{-1}}K}(R) \leq N_{\chi \varphi}$. Because $Q$ is receptive, $\chi \varphi$ extends to $\psi':  N_{\chi \varphi} \rightarrow N_S(Q)Q$. The restriction of $\psi'$ to $N_S^{{}^{\varphi^{-1}}K}(R)$ is the morphism $\psi$ required.
\end{proof}

\begin{Lemma}[cf.~\cite{BLO2}]\label{Lemma10}
Let $\Ff$ be a pro-fusion system on $S$, $Q \leq S$ and $K \leq \Aut_{\Ff}(Q)$. Suppose that there is a subgroup in the $\Ff$-isomorphism class of $Q$ that is fully normalized in~$\Ff$. Then $Q$ is fully $K$-normalized in~$\Ff$ if and only if all $\varphi\in\Hom_\Ff(N_S^K(Q)Q,S)$ satisfy ${\varphi(N_S^K(Q))=N_S^{^\varphi K}(\varphi(Q))}$.

If $Q$ is an open subgroup of $S$, the previous properties are equivalent to $|N^K_S(Q)Q/Q|$ being maximal in the $\Ff$-isomorphism class of $Q$, i.e.~in the set $\{|N^{^\chi K}_S(\chi(Q))\chi(Q)/\chi(Q)|\mid\chi\in\Hom_\Ff(Q,S)\}$ (this set is bounded, by Lemma~\ref{la:open}).
\end{Lemma}

\begin{proof}
Suppose that $Q$ is fully $K$-normalized in~$\Ff$ and $\varphi\in\Hom_\Ff(N_S^K(Q)Q,S)$. Clearly $\varphi$ restricts to a map $ N_S^K(Q) \rightarrow N_S^{{}^{\varphi}K}(\varphi(Q))$. By Lemma~\ref{la:equiv_fnorm} applied to $\varphi^{-1}|_{\varphi(Q)}: \varphi(Q) \rightarrow Q$, there is a map $\psi : N_S^{{}^{\varphi}K}(\varphi(Q))\varphi(Q) \rightarrow N_S^K(Q)Q$. This takes $\varphi(Q)$ to $Q$, hence restricts to a map $N_S^{{}^{\varphi}K}(\varphi(Q)) \rightarrow N_S^K(Q)$.
Thus there are morphisms in both directions between $N_S^K(Q)$ and $N_S^{\tilde\varphi(K)}(\varphi(Q))$, and so, by Lemma~\ref{EI}, $\varphi(N_S^K(Q))=N_S^{^\varphi K}(\varphi(Q))$.

Conversely, suppose that for all $\varphi: N_S^K(Q)Q\to S$ we have $\varphi(N_S^K(Q))=N_S^{^\varphi K}(\varphi(Q))$. Let $\theta\in\Hom_\Ff(Q,S)$ be such that $P:=\theta(Q)$ is fully normalized in~$\Ff$. By Lemma~\ref{Lemma11}, there is a
 $\chi\in\Aut_\Ff(P)$ such that $P$ is fully ${}^\chi ({}^\theta K)$-normalized. Apply Lemma~\ref{la:equiv_fnorm} to $\eta:=\chi \theta |_Q$ to obtain a morphism $\psi : N_S^K(Q)Q \rightarrow N_S^{{}^\eta K}(P)P$ such that $\psi (Q)=P$ and ${}^\eta K = {}^\psi K$. The restriction  $\psi |_{N_S^K(Q)} : N_S^K(Q) \rightarrow N_S^{{}^\eta K}(P)$ must be onto, by hypothesis, hence the induced homomorphism $\Aut_S^K(Q) \rightarrow \Aut _S^{{}^\eta K}(P)$ must be onto. But we have arranged for $\Aut _S^{{}^\eta K}(P)$ to be a Sylow pro-$p$ subgroup of $\Aut _\Ff ^{{}^\eta K}(P)$, so $\Aut_S^K(Q)$ must be a Sylow pro-$p$ subgroup of $\Aut_\Ff^K(Q)$, showing that $Q$ is fully $K$-automized.

It is easy to check that $Q$ is receptive in $\Ff$. Indeed, since $P$ is fully normalized in~$\Ff$, Lemma~\ref{la:equiv_fnorm} shows that there is a $\psi\in\Hom_\Ff(N_S(Q),N_S(P))$ such that $\psi(Q)=P$. If $\chi\in\Iso_{\Ff}(R,Q)$, then $N_\chi\le N_{\psi\chi}$ and there is a $\rho\in\Hom_\Ff(N_{\psi\chi},N_S(P))$ with $\rho(R)=P$. Moreover, $\rho(N_\chi)\le\psi(N_S(Q))$, thus $\psi^{-1}\rho|_{N_\chi}$ extends $\chi$.

If $Q$ is an open subgroup of $S$, then Lemma~\ref{la:open} shows that for every subgroup $P$ and $\varphi\in\Hom_\Ff(Q,P)$ the image $P$ is also an open subgroup of $S$ and that $|QN^{^\varphi K}_S(P)/P|$ is finite and bounded by $|S:Q|$.  For any $\varphi\in\Hom_\Ff(N_S^K(Q)Q,S)$ we have $\varphi(N_S^K(Q)Q)\le N_S^{^\varphi K}(P)P$ and $\varphi(N_S^K(Q)Q)/Q\le N_S^{^\varphi K}(P)\varphi(Q)/P$. If $|N^K_S(Q)/Q|$ is maximal then we must have equality.

Conversely, suppose that $Q$ is fully $K$-normalized in~$\Ff$ and $\chi \in \Iso_{\Ff}(Q,P)$. Applying Lemma~\ref{la:equiv_fnorm} to $\chi^{-1}$ yields a $\psi :N_S^{{}^{\chi}K}(P)P \rightarrow N_S^K(Q)Q$, which restricts to an isomorphism $P \rightarrow Q$. This induces a morphism $N_S^{{}^{\chi}K}(P)P/P \rightarrow N_S^K(Q)Q/Q$, so $|N_S^K(Q)Q/Q|$ is maximal.
\end{proof}

\begin{Lemma}\label{Lemma16}
Let $\Ff$ be a pro-fusion system on $S$, $Q \leq S$ and $K \leq \Aut _\Ff(Q)$. If $Q$ is fully $K$-normalized in~$\Ff$, then for any $\varphi\in\Hom_\Ff(N_S^K(Q),S)$ the subgroup $\varphi(Q)$ is fully ${}^\varphi K$-normalized in~$\Ff$.
\end{Lemma}

\begin{proof}
Given $\theta \in \Hom _\Ff( N_S^{{}^\varphi K}(\varphi (Q)) \varphi (Q),S)$, compose it with $\varphi$. By Lemma~\ref{Lemma10}, the map $\theta \varphi : N_S^K(Q) \rightarrow N_S^{{}^{\theta \varphi}K}( \theta \varphi (Q))$ is onto, so $\theta : N_S^{{}^{ \varphi}K}( \varphi (Q)) \rightarrow N_S^{{}^{\theta \varphi}K}( \theta (\varphi (Q)))$ is also onto. By Lemma~\ref{Lemma10} again, $\varphi(Q)$ is fully ${}^\varphi K$-normalized.
\end{proof}

When $\Ff$ is a pro-fusion system on $S$ and $N$ is an open strongly closed subgroup of $S$ we can form the quotient system $\Ff /N$ in the manner of \ref{lemma4}; it is a fusion system on $S/N$.
Just as in the case of a fusion system on a finite $p$-group, the image of a fully $\Ff$-normalized subgroup is a fully $\Ff/N$-normalized subgroup.

\begin{Lemma}\label{Lemma14}
Let $\Ff$ be a pro-fusion system on a pro-$p$ group $S$. Let $N$ be an open and strongly $\Ff$-closed subgroup of $S$ and let $Q$ be a subgroup of $S$ containing $N$ that is fully normalized in~$\Ff$. Then $Q/N$ is fully normalized in~$\Ff/N$.
\end{Lemma}
\begin{proof}
For simplicity, write $\oline\Ff:=\Ff/N$ and $\oline X:=XN/N$ for all subgroups $X$ of $S$.
First, note that, since $N\le Q$, we have $\overline{N_S(Q)}=N_{\overline S}(\oline Q)$. By Lemma~\ref{Lemma10}, we only need to show that for every $\oline \varphi\in\Hom_{\overline \Ff}(N_{\overline S}(\oline Q),\oline S)$ we have $\oline \varphi(N_{\overline S}(\oline Q))=N_{\overline S}(\oline \varphi(\oline Q))$. The inclusion of the left term in the right one is clear, so it is enough to show that $|\oline \varphi(N_{\overline S}(\oline Q))|\ge |N_{\overline S}(\oline \varphi(\oline Q))|$. Since $\oline \varphi$ is a morphism of $\oline \Ff$, there is a $\varphi\in\Hom_\Ff(N_S(Q),S)$ inducing $\oline \varphi$. By Lemma~\ref{la:equiv_fnorm}, there is a morphism $\theta \in \Hom_\Ff(N_S(\varphi(Q)),N_S(Q))$ that extends $(\varphi|_Q)^{-1}$ up to an automorphism of $Q$. This induces a morphism in $\Hom_{\overline \Ff}(N_{\overline S}(\oline \varphi(\oline Q)),N_{\overline S}(\oline Q))$, and so $|\oline \varphi(N_{\overline S}(\oline Q))|\ge |N_{\overline S}(\oline \varphi(\oline Q))|$.
\end{proof}

As a consequence we obtain:

\begin{Corollary}\label{QuotientSaturation}
Let $\Ff$ be a saturated pro-fusion system on a pro-$p$ group $S$, and let $N$ be an open and strongly $\Ff$-closed subgroup of $S$. Then $\Ff/N$ is saturated fusion system on $S/N$.
\end{Corollary}
\begin{proof}
Let $\oline Q$ be a subgroup of $S/N$ and let $Q$ be its inverse image in $S$. Since $\Ff$ is saturated, there exists a subgroup $R$ in the $\Ff$-isomorphism class of $Q$ that is fully normalized in~$\Ff$. By Lemma \ref{Lemma14}, the subgroup $\oline R:=R/N$ is fully normalized in~$\Ff/N$. Since $\oline R$ is in the $\Ff/N$-isomorphism class of $\oline Q$, this implies that $\Ff/N$ is saturated.
\end{proof}

\section{Morphisms of pro-fusion systems}

Because we have defined a pro-fusion system as an inverse limit of fusion systems, it is incumbent upon us to deal carefully with morphisms and quotients. A morphism between two pro-fusion systems is defined exactly as for fusion systems, that is to say as a homomorphism, $\alpha$, between the underlying pro-$p$ groups that satisfies certain conditions which can be expressed by the existence of a certain functor, $A$, just as in Definition~\ref{def:mor}. We now have a category of pro-fusion systems. It is easy to verify that if $\Ff = \ilim_{i \in I} \Ff_i$ is a pro-fusion system then $\Ff$ really is the inverse limit in this category.

Everything in Paragraph~\ref{lemma4} is still valid for pro-fusion systems, provided that $Q$ is taken to be open.

It might appear that this definition of morphism does not capture all the information in the inverse limit system, but we will show that, in fact, it does so in the strongest possible reasonable way. We phrase this as a continuity condition, analogous to the characterization of continuity for maps of profinite sets or groups. It is, in fact, equivalent to a morphism of the underlying pro-systems, as the reader who is familiar with such things will notice.

\begin{Definition}
\label{def:cont}
Let $\Ff = \ilim_{i \in I} \Ff_i$ and $\Gg = \ilim_{j \in J}\Gg_j$ be pro-fusion systems, with canonical sets of morphisms $\{(f_i,F_i):\Ff\rightarrow\Ff_i\mid i\in I\}$ and $\{(g_j,G_j):\Gg\rightarrow\Gg_j\mid j\in J\}$. A morphism $(\alpha,A) : \Ff \rightarrow \Gg$ is said to be {\it continuous} if for each $j \in J$ there is an $i \in I$ such that the composition $\Ff \stackrel{(\alpha,A)}{\rightarrow} \Gg \stackrel{(g_j,G_j)}{\rightarrow} \Gg_j$ factors through $\Ff \stackrel{(f_i,F_i)}{\rightarrow} \Ff_i$.
\end{Definition}

This definition appears to depend on the ways in which $\Ff$ and $\Gg$ are expressed as inverse limits (but see Corollary~\ref{cor:indeplim}).

If $\Ff= \ilim_{i\in I} \Ff_i$, we can consider the fusion systems $\langle F_i(\Ff) \rangle \subseteq \Ff_i$. These can be assembled into an inverse limit system and we can consider $\ilim_{i\in I} \langle F_i(\Ff) \rangle$.

\begin{Lemma}
\label{la:finim}
Suppose that $\Ff = \ilim_{i\in I} \Ff_i$ is a pro-fusion system. Then for each $i \in I$ there is a $j \geq i$ such that $F_i(\Ff) = F_{i,j}(\Ff_j)$.
\end{Lemma}

\begin{proof}
For each $j \geq i$ consider $X_j = \langle F_{i,j}(\Ff_j) \rangle \subseteq \Ff_i$. Suppose that $\varphi \in \cap_{j \geq i} X_j$ and let $Y_j = \{ \theta \in \Ff_j \mid F_{i,j}(\theta)=\varphi \}$. These $Y_j$ form an inverse system of finite sets, so their inverse limit is non-empty. An element of the inverse limit defines an element $\tilde{\varphi} \in \Ff$ such that $F_i(\tilde{\varphi})=\varphi$, and we conclude that $F_i(\Ff) = \cap_{j \geq i} X_j$.
Since the sets $X_i$ are finite, we must have $F_i(\Ff)= F_{i,j}(\Ff_j)$ for some $j \geq i$.
\end{proof}

\begin{Lemma}
\label{la:angle}
The identity map on $S$ induces an isomorphism between $\Ff$ and $\ilim_{i\in I} \langle F_i(\Ff) \rangle$. It is continuous in both directions.
\end{Lemma}

\begin{proof}
It is easy to see that we have an isomorphism and that the morphism $\ilim_{i\in I} \langle F_i(\Ff) \rangle \rightarrow \Ff$ is continuous.

For the other direction, consider some fixed morphism $\Ff \stackrel{F_i}{\rightarrow} \Ff_i$. Choose $j \geq i$ as in Lemma~\ref{la:finim} so that $F_i(\Ff) = F_{i,j}(\Ff_j)$. Thus $F_{i,j}$ has image in $\langle F_i(\Ff) \rangle$, and so $F_i$ factors through $\Ff_j$.
\end{proof}

\begin{Proposition}
\label{thm:cont}
Every morphism between pro-fusion systems is continuous.
\end{Proposition}

\begin{proof}
Suppose that $\Ff = \ilim_{i\in I} \Ff_i$ and $\Gg = \ilim _j \Gg_j$ are a pro-fusion systems on $S$ and $T$. By Lemma~\ref{la:angle}, we may assume that $\Ff_i = \langle F_i(\Ff) \rangle$. Let $\alpha:S \rightarrow T$ be a group homomorphism that, together with a functor $A$, induces a morphism $\Ff \rightarrow \Gg$. Fix some $\Gg_j$; since $\alpha$ is continuous, there is an $i$ such that $S \stackrel{\alpha}{\rightarrow} T \stackrel{g_j}{\rightarrow} T_j$ factors through $S \stackrel{f_i}{\rightarrow} S_i$. We need to construct the functor $A_{i,j}$ associated to the factoring map $S_i \rightarrow T_j$; it is sufficient to define this on elements of $F_i(\Ff)$.

Given $\varphi \in F_i(\Ff)$, we have $\varphi = F_i(\tilde{\varphi})$ for some $\tilde{\varphi} \in \Ff$. We define $A_{i,j}(\varphi)=G_jA(\tilde{\varphi})$; it is straightforward to check that this has the right properties.
\end{proof}

The continuity of morphisms has various basic consequences, such as the following ones.

\begin{Corollary}
\label{cor:indeplim}
If a pro-fusion system $\Ff$ can be defined up to isomorphism as an inverse limit in two different ways, say as $\ilim_{i \in I} \Ff_i$ or as $\ilim_{i \in I'} \Ff'_i$, then the isomorphism is continuous. Therefore, the continuity of a morphism of pro-fusion systems does not depend on the representations of the pro-fusion systems as inverse limits.
\end{Corollary}

\begin{Corollary}
Let $\Ff$ be a pro-fusion system on a a pro-$p$ group $S$ and let $P$ and $Q$ be subgroups of $S$. Then the topology on $\Hom_{\Ff}(P,Q)$ depends only on the isomorphism class of $\Ff$.

If $(\alpha,A):\Ff \rightarrow \Gg$ is a morphism of pro-fusion systems, then the group homomorphism $\Aut_{\Ff}(P) \rightarrow \Aut_{\Gg}(\alpha(P))$ induced by $A$ is a continuous homomorphism of profinite groups and the induced map $\Hom_{\Ff}(P,Q) \rightarrow \Hom _{\Gg}(\alpha(P),\alpha(Q))$ is a continuous map of profinite sets.
\end{Corollary}

\begin{Proposition}
\label{prop:fbar}
For any pro-fusion system $\Ff$ on $S$ we have $\Ff \cong \ilim _{N \osc} \langle\oline\Ff_N \rangle$, where the limit is taken over the set of open strongly closed subgroups of $S$ ordered by inclusion.
\end{Proposition}

\begin{proof}
There is a natural morphism $\Ff \to \langle\oline\Ff_N \rangle$. The image of $\Ff$ in $\langle\oline\Ff_N \rangle$ is $\oline\Ff_N$, so the image of $\Ff$ in $\ilim _{N \osc} \langle\oline\Ff_N \rangle$ is $\ilim _{N \osc} \oline\Ff_N $, by a standard result about profinite sets, and $\Ff$ is isomorphic to its image.

Suppose that $\Ff= \ilim_{i\in I} \Ff_i$ and let $N_i=\Ker (f_i)$, for all $i\in I$. By Lemma~\ref{la:angle}, we may assume that $\Ff_i = \langle F_i(\Ff) \rangle$. Let $X$ be the set $\{N_i\mid i\in I\}$; then $X$ is cofinal in the set of all open strongly closed subgroups of $S$, so we may replace $\ilim_{N \osc}$ by $\ilim_{N \in X} $.

Given $M \in X$, we know that $M=N_i$ for some $i$, so $S/M \cong f_i(S)$. By Lemma~\ref{la:finim}, there is a $j \geq i$ such that $F_{i,j}(\Ff_j) = F_i(\Ff)$.

We have a commutative diagram of groups
\[
\begin{CD}
S @>>> S/N_j @>>> S/M \\
@|     @V{\equiv}VV @V{\equiv}VV \\
S @>{f_j}>> f_j(S) @>{f_{i,j}}>> f_i(S),
\end{CD}
\]
     in which the arrows on the top row are the quotient maps and the vertical isomorphisms are uniquely determined. This extends to a diagram of functors
\[
\begin{CD}
\Ff @>>> \langle\oline\Ff_{N_j} \rangle @>>> \langle\oline\Ff_M \rangle \\
@|     @V{\equiv}VV @V{\equiv}VV \\
\Ff @>{F_j}>> \Ff_j @>{F_{i,j}}>> \Ff_i,
\end{CD}
\]
in which the vertical maps are isomorphisms, because the domain and codomain are both defined in the same way, as the fusion system generated by the morphisms between the subgroups of $S/N_j$ or of $S/M$ that are induced by morphisms in $\Ff$. Thus the image of the morphism $\langle\oline\Ff_{N_j} \rangle \to \langle\oline\Ff_M \rangle $ is $\oline\Ff_M$.

This forces $\ilim _{N \in X} \langle\oline\Ff_N \rangle = \ilim _{N \in X} \oline\Ff_N$.
\end{proof}

\npar
\label{countbase}
There are many equivalent characterizations of countably based profinite groups. The two that will be useful to us are that there are only countably many open subgroups and that the group can be expressed as an inverse limit of finite groups indexed by the set $\mathbb N$ of natural numbers ordered in the usual way. Any finitely generated profinite group is countably based.

We define a pro-fusion system to be {\it countably based\/} if it is isomorphic to an inverse limit of fusion systems indexed by the natural numbers with the usual ordering.

\begin{Lemma}
A pro-fusion system $\Ff$ on the pro-$p$ group $S$ is countably based if and only if $S$ is countably based.
\end{Lemma}

\begin{proof}
If $\Ff$ is countably based then so is $S$, directly from the definitions.

Conversely,  by Proposition~\ref{prop:fbar} we know that $\Ff \cong \ilim _{N \osc} \langle\oline\Ff_N \rangle$. If $S$ is countably based then it has only countably many open subgroups, so it certainly has only countably many open strongly closed subgroups; enumerate them as $N_1,N_2, \ldots $.

Recursively, choose $M_1=N_1$ and $M_r=M_{r-1} \cap N_r$.   Then $\{ M_r \} _{r \in \mathbb N}$ is ordered in the same way as $\mathbb N$ and is cofinal in the set of open strongly closed subgroups of $S$. Hence $\Ff \cong \ilim _{r \in \mathbb N} \langle\oline\Ff_{M_r} \rangle$.
\end{proof}

\section{Pro-saturated fusion systems}

We introduce pro-saturated fusion systems and show that they are well behaved.

\begin{Definition}\label{ProFusionSystems}
We say that a pro-fusion system is a {\it pro-saturated fusion system} if it is isomorphic to an inverse limit of saturated fusion systems on finite $p$-groups.
\end{Definition}

\begin{Lemma}
\label{la:satim}
Suppose that $\Ff= \ilim_{i\in I} \Ff_i$, where the $\Ff_i$ are saturated fusion systems. Then $F_i(\Ff)$ is  a saturated fusion system for all $i\in I$, and $\Ff \cong \ilim_{i\in I} F_i(\Ff)$.
\end{Lemma}

\begin{proof}
From Lemma~\ref{la:finim} we know that $F_i(\Ff)=F_{i,j}(\Ff_j)$ for some $j \geq i$. By Theorem~\ref{thmPuig}, $F_{i,j}(\Ff_j)$ is a saturated fusion system and thus, by Lemma~\ref{la:angle}, $\Ff \cong \ilim_{i\in I} F_i(\Ff)$.
\end{proof}

\begin{Lemma}\label{propuig0}
Let $\Ff= \ilim_{i\in I} \Ff_i$ be a pro-saturated fusion system.
Then for each $\varphi\in\Hom_\Ff(P,Q)$ and each $i\in I$ there exists a $\tilde\varphi \in \Hom_\Ff(f_i^{-1}f_i(P), f_i^{-1}f_i(Q))$ such that $F_i(\tilde\varphi)=F_i(\varphi)$.
\end{Lemma}

\begin{proof}
Let $\Ff= \ilim_{i\in I} \Ff_i$ with the $\Ff_i$ saturated.
For each $j \geq i$ let $X_j$ be the set of the morphisms in $\Hom_{\Ff_j}(f_{i,j}^{-1}f_i(P), f_{i,j}^{-1}f_i(Q))$ with image $F_i(\varphi)$ in $\Ff_i$. Then each $X_j$ is a finite set, and it is non-empty, by Theorem~\ref{thmPuig0}. An inverse limit of non-empty finite sets is non-empty, and an element of $\ilim_{j \geq i} X_j$ determines a morphism  $\tilde\varphi \in \Hom_\Ff(f_i^{-1}f_i(P), f_i^{-1}f_i(Q))$.
\end{proof}

Given a pro-fusion system $\Ff$ on $S$ and an open strongly closed subgroup $N$ of $S$ we can form the fusion system $\Ff/N$ just as in the finite case.

\begin{Proposition}
\label{propuig1}
Let $\Ff$ be a pro-saturated fusion system on $S$ and let $N$ be an open strongly $\Ff$-closed subgroup of $S$. For any $\varphi\in\Hom_\Ff(P,Q)$ there exists a $\tilde\varphi \in \Hom_\Ff(PN,QN)$ such that $\varphi$ and $\tilde\varphi$ induce the same homomorphism $PN/N \to QN/N$. As a consequence, $\langle\oline\Ff_N \rangle = \Ff /N$, and so there is a natural morphism $\Ff \to \Ff/N$.
\end{Proposition}

\begin{proof}
If $N$ is one of the $N_i$, then this is just a reformulation of Lemma~\ref{propuig0}. Otherwise, there is some $N_i \leq N$, and we can use the $\tilde\varphi$ produced by using Lemma~\ref{propuig0} relative to $\Ff_i$.

The last part follows from the first part and the definition of $\Ff/N$.
\end{proof}

\begin{Proposition}\label{Proposition7} If $\Ff$ is a pro-saturated fusion system on $S$, then $\Ff\cong\ilim_{N \osc}\Ff/N$ as $N$ runs through all open strongly $\Ff$-closed subgroups of $S$ ordered by inclusion.
\end{Proposition}

\begin{proof}
This follows from Proposition~\ref{prop:fbar} and Lemma~\ref{propuig1}.
\end{proof}

\npar We denote the restriction of a pro-fusion system $\Ff$ on $S$ to the open subgroups of $S$ by $\Ff^o$. If $\Ff$ is pro-saturated then we can tell whether an open subgroup $N$ of $S$ is strongly $\Ff$-closed or not just by considering $\Ff^o$. For there is some $N_i \leq N$ and, if $P \leq N$ and $\varphi(P) \not \leq N$, then Proposition~\ref{propuig1} shows that $PN_i \leq N$ but $\tilde \varphi (PN_i) \not \leq N$. It is clear from the definition that $\Ff/N$ only depends on $\Ff^o$.

From these considerations and using Proposition~\ref{Proposition7}, we see that if $\Ff$ and $\Gg$ are both pro-saturated fusion systems on $S$ and $\Ff^o=\Gg^o$ then $\Ff=\Gg$. In a sense, pro-saturated fusion systems are completely determined by what happens on open subgroups.

\npar If $\Ff$ is a pro-fusion system on $S$ such that $\Ff^o$ is saturated and $N$ is an open strongly closed subgroup of $S$, then $\Ff/N$ is saturated, as follows directly from the definitions. If $M \leq N$ is also open and strongly closed, then $\Ff/N = (\Ff/M)/(N/M)$, so there is a natural morphism $\Ff/M \to \Ff/N$, by Theorem~\ref{thmPuig}. We set $\Ff_{\ps} = \ilim _{N \osc} \Ff/N$. Then $\Ff_{\ps}$ is pro-saturated and there is a natural morphism $\Ff_{\ps} \to \Ff$ with the property that $\Ff_{\ps}^o = \Ff^o$.

\begin{Example}
Let $p$ be an odd prime and let $S$ be a free pro-$p$ group on two generators $x$ and~$y$. Let $C$ be a cyclic group of order 2 and let it act on $S$ by $x \leftrightarrow y$. Set $G = S\rtimes C$. There are two saturated fusion systems on $S$, $\Ff_S(S)$ and $\Ff_S(G)$: we want to mix them. Define $\Ff$ to be the subcategory of $\Ff_S(G)$ in which the morphisms with cyclic domain are those from $\Ff_S(G)$ and the morphisms with non-cyclic domain are those from $\Ff_S(S)$. We need to check that $\Ff$ is indeed a pro-fusion system.

There is a chain of open subgroups $\{ N_i | i\in I \}$ in $S$, each normal in $G$, such that $S = \ilim_{i\in I} N_i$. Set $G_i=(S/N_i) \rtimes C$. On $S/N_i$ we define a fusion system $\Ff_i$ in an analogous way to $\Ff$. Then $\Ff = \ilim_{i\in I} \Ff_i$.

It is easily verified that $\Ff$ is saturated using the fact that in a free pro-$p$ group the normalizer of a cyclic group is cyclic (just as in the case of an abstract free group).

Also $\Ff_{\ps} = \Ff_S(S)$, so $\Ff$ is not pro-saturated.
\end{Example}

\section{Saturation of pro-saturated fusion systems}

We show that a pro-saturated fusion system is saturated, at least if it is countably based.

\begin{Theorem}\label{OpenSatAx}
If $\Ff$ is a pro-saturated fusion system, then $\Ff^o$ is saturated.
\end{Theorem}

\begin{proof}
Suppose that we are given an open subgroup $Q\le S$. Notice that $\{N_i\mid N_i<Q\}$ is cofinal in $\{N_i | i \in I \}$, so $\Ff\simeq\ilim_{N_i<Q}\Ff_i$. Thus we can assume that all the $N_i$ are subgroups of $Q$. For each $N_i<Q$ let $X_i:=\{\varphi\in\Hom_{\Ff_i}(f_i(Q),S_i)\mid \varphi(f_i(Q))\text{ is fully normalized in~$\Ff_i$}\}$. Since $\Ff_i$ is a saturated fusion system and $S_i$ is finite, $X_i$ is a finite non-empty set. By Lemma~\ref{Lemma14}, for every $N_j<N_i$ there is a natural map $X_j\to X_i$; thus $\ilim_{N_i<Q} X_i\ne\emptyset$. An element of the limit determines a map $\psi:Q\to S$ such that $f_i(\psi(Q))$ is fully normalized in~$\Ff_i$ for each $N_i<Q$.

We show that $P=\psi(Q)$ is fully normalized in~$\Ff$. Given $\varphi:R\to P$, it induces $\oline\varphi_i:R/N_i\to P/N_i$, which, given that $P/N_i$ is fully normalized in~$\Ff_i$, extends to $\rho_i:N_{\overline\varphi _i}\to N_{S/N_i}(Q/N_i)$. For $N_j<N_i$ there is a natural map $N_{\overline\varphi_j}\to N_{\overline\varphi_i}$; hence any extension $\rho_j$ induces an extension $\rho_i$. Let $X_i$ be the set of all possible extensions of $\oline\varphi_i$ to $N_{\overline\varphi_i}$. These are finite non-empty sets and there are natural maps $X_j\to X_i$ for $N_j<N_i$. Thus $\ilim_{N_i<Q}X_i$ is non-empty and an element in this set induces a map $\ilim_{N_i<Q}N_{\overline\varphi_i}\to \ilim_{N_i<Q}N_{S/N_i}(P/N_i)$, i.e. a map $N_\varphi\to N_S(P)$.

The image of $N_S(P)$ in $S/N_i$ is $N_{S/N_i}(P/N_i)$, and this maps to a Sylow $p$-subgroup of $\Aut_{\Ff_i}(P/N_i)$. Since this is true for all $i \in I$, $\Aut_S(P)$ is dense in a Sylow $p$-subgroup of $\Aut_\Ff(P)$. But $\Aut_S(P)$ is closed, so it is a Sylow $p$-subgroup of $\Aut_\Ff(P)$. Thus, $\Ff^o$ satisfies the saturation axioms.
\end{proof}

\begin{Theorem}\label{prosatsat}
If $\Ff$ is a countably based pro-saturated fusion system then $\Ff$ is saturated.
\end{Theorem}

Before the proof we need a lemma.

\begin{Lemma}\label{Lemma17}
Let $\Ff$ be a pro-saturated fusion system on $S$ and $P\le S$, and let $P_n \leq P_{n-1} \leq \dots \leq P_1$ be a non-empty sequence of open, strongly closed subgroups of $S$. Then there exists a subgroup $Q$ in the $\Ff$-isomorphism class of $P$ such that $QP_i$ is fully normalized in~$\Ff$ for each $i \in \{1,\dots,n\}$.
\end{Lemma}

\begin{proof}
We use induction on $n$. The result is true if $n=1$ by Theorem~\ref{OpenSatAx}. Suppose that $n\ge 2$ and that there exists a subgroup $R$ that is $\Ff$-isomorphic to $P$ and such that $RP_i$ is fully normalized in~$\Ff$ for each $2\le i\le n$. By Theorem~\ref{OpenSatAx}, there exists $\varphi \in \Hom _\Ff ( RP_1, S)$ such that $\varphi (RP_1)$ is fully normalized in~$\Ff$. By Lemma~\ref{la:equiv_fnorm}, there is a  $\psi \in \Hom_\Ff (N_S(RP_1), N_S(\varphi(RP_1))$ such that $\varphi(RP_1) = \psi(RP_1) = \psi(R)P_1$. Let $Q:=\psi(R)$. Since $QP_1=(QP_i)P_1$, we find that $N_S(QP_i)\le N_S(QP_1)$, so each $N_S(QP_i)$ lies in the domain of $\psi$. Lemma~\ref{Lemma16} shows that each $RP_i$ is fully normalized in~$\Ff$.
\end{proof}

\begin{proof}[Proof of Theorem~\ref{prosatsat}]

Since $\Ff$ is countably based, we can write $\Ff= \ilim_{i \in \mathbb N} \Ff_i$, and the $P_i:=\Ker(f_i)$ form a sequence of open, strongly closed subgroups of $S$. Given a subgroup $P \leq S$, let~$X_i$ be the set of $\varphi:PP_i/P_i\to S/P_i$ such that $\varphi(P)P_i/P_i$ is fully normalized in~$\Ff/P_i$. By Lemma~\ref{Lemma17}, $X_i$ is finite and non-empty for all $i$. Moreover, there is a natural map $X_{i+1}\to X_i$ induced by the canonical projection $PP_i/P_i\simeq P/(P\cap P_i)\to P/P\cap P_{i+1}\simeq PP_{i+1}/P_{i+1}$. Thus $\ilim_i X_i$ is non-empty, and an element $\psi\in\ilim_i X_i$ induces a map $\psi:P\to S$ such that $\psi(P)P_i/P_i$ is fully normalized in~$\Ff/P_i$ for each $i$. Set $R:=\psi(Q)$; we need to show that $R$ is fully normalized in~$\Ff$.

First we show that every $\varphi:Q\to R$ extends to $N_\varphi$. Since $RP_i/P_i$ is fully normalized in~$\Ff_i$, we know that $\oline\varphi_i:QP_i/P_i\to RP_i/P_i$ extends to a morphism $N_{\overline\varphi_i}\to N_{S/P_i}(R/P_i)$. Let~$X_n$ be the set of all maps $N_\varphi\to N_{S/P_n}(R/P_n)$ of the form $N_\varphi\to N_{\overline\varphi_i}\to N_{S/P_i}(R/P_i)\to N_{S/P_n}(R/P_n)$ for some $i\ge n$. $X_n$ is clearly non-empty; it is also finite, since all the maps have to have $P_n$ in the kernel. Moreover, there is a natural map $X_{n+1}\to X_n$ induced by the projection $N_{S/P_{n+1}}(R/P_{n+1})\to N_{S/P_n}(R/P_n)$. Thus $\ilim_n X_n$ is non-empty and an element of this limit yields $\tilde\varphi:N_\varphi\to N_S(R)$, extending $\varphi$.

Next we show that $\Aut_S(R)$ is a Sylow $p$-subgroup of $\Aut_\Ff(R)$. This is done by applying the following lemma to the family $N_{S/P_i}\to\Aut_{\Ff_i/P_i}(RP_i/P_i)$.
\end{proof}

\begin{Lemma}\label{Lemma18}
Let $\{G_i\}_{i\in I}$ be a directed system of finite groups and $\{S_i\}_{i\in I}$ a directed system of finite $p$-groups. Let $\{\chi_i:S_i\to G_i \}_{i\in I}$ be a map of directed systems such that $\chi_i(S_i)$ is a Sylow $p$-subgroup of $G_i$ for all $i$. Then $\chi(\ilim_{i\in I} S_i)$ is a Sylow $p$-subgroup of $\ilim_{i\in I} G_i$.
\end{Lemma}

\begin{proof}
Set $G = \ilim_{i\in I} G_i$ and $S = \ilim_{i\in I} S_i$. It is easy to see that $\chi: S \to G$ is injective.

Let $T$ be a Sylow pro-$p$ subgroup of $G$ such that $\chi(S) \le T$. For each $i$ let $p_i:G\to G_i$ be the canonical projection. Then there exists a $g_i\in G_i$ such that $^{g_i}p_i(T)\le \chi_i(S_i)$. Let $X_i$ be the set of all such $g_i$. There is a natural map $X_i\to X_j$ sending $g_j$ to $p_{i,j}(g_i)$. Hence $\ilim_{i\in I} X_i$ is non-empty, and an element of it yields an element $g \in G$ such that $^gT\le\chi(S)$. Given that $\chi(S)\le T$ we obtain $\chi(S)=T$.
\end{proof}

We do not know whether the hypothesis that the system should be countably based is necessary.

\section{Invariant subsystems and a relative form of Alperin's Fusion Theorem}

Let $\Ff$ be a pro-fusion system on a $p$-group $S$. We introduce here the notion of a $T$-subsystem, where $T$ is a strongly $\Ff$-closed subgroup of $S$. This generalizes the notion of an $\Ff$-invariant system in the literature (\cite{aschbacher:NormalSubsystems, craven, Linckelmann}), by retaining as objects in the $T$-subsystem all subgroups of $S$, not just the ones contained in $T$. If $G$ is a finite group with Sylow $p$-subgroup~$S$ and $H$ is a normal subgroup of $G$ then the morphisms between the subgroups of $S$ that are induced by conjugation by an element of $H$ will form an $(H\cap S)$-subsystem of $\Ff_S(G)$. So we are attempting to model a normal subgroup.

\begin{Definition}\label{subsystem}
Let $\Ff$ be a fusion system on a $p$-group $S$ and let $T$ be a strongly $\Ff$-closed subgroup of $S$. A {\it $T$-subsystem\/} $\Ee$ of $\Ff$ is a subcategory of $\Ff$ on the same objects, with morphisms satisfying:
\begin{enumerate}[(a)]
\item $\Ee$ contains every group homomorphism induced by conjugation by an element of $T$,
\item any morphism in $\Ee$ factors as an isomorphism of $\Ee$ followed by an inclusion,
\item for all $P,P'\le Q\le S$ and $\varphi\in\Hom_\Ff(Q,S)$ we have $\varphi |_{P'} \Hom_\Ee(P,P')=\Hom_\Ee(\varphi(P),\varphi(P'))\varphi |_P$,
\item for all $Q\le S$ and $\psi\in\Hom_\Ee(Q,S)$ we have $\psi(u)u^{-1}\in T,\,\forall u\in Q$.
\end{enumerate}
\end{Definition}

Note that the full subcategory of $\Ee$ on the subgroups of $T$ is what is known as an $\Ff$-invariant fusion system in the literature.

A morphism between a $T_1$-subsystem $\Ee_1$ of $\Ff_1$ and a $T_2$-subsystem $\Ee_2$ of $\Ff_2$ is a morphism $(\alpha, A):\Ff_1 \to \Ff_2$ such that $\alpha (T_1) \leq T_2$ and $A(\Ee_1) \subseteq \Ee_2$.

In an analogous way to the saturation of $\Ff$ we define the saturation of $\Ee$.

\begin{Definition}
We say that $Q$ is {\it receptive in $\Ee$\/} if for all $\varphi\in\Iso_\Ee(R,Q)$ there exists $\tilde \varphi :N^T_\varphi\to N_T(Q)Q$ in $\Ee$ such that $\tilde \varphi |_R=\varphi$, where
$$N^T_\varphi=\{x\in N_T(R)R \mid \exists y\in N_T(Q)Q, \varphi(xux^{-1})=y\varphi(u)y^{-1}, \forall u\in R\}\,.$$
\end{Definition}

\begin{Definition}
Let $K\le\Aut_\Ee(Q)$. We say that $Q$ is {\it fully $K$-automized in $\Ee$} if $\Aut_T(Q)\cap K$ is a Sylow $p$-subgroup of $\Aut_\Ee(Q)\cap K$. We say that $Q$ is {\it fully $K$-normalized in~$\Ee$} if $Q$ is receptive and fully $K$-automized in $\Ee$.
\end{Definition}

When $K=\Aut_\Ee(Q)$ we say that $Q$ is {\it fully normalized in~$\Ee$} rather that fully ${\Aut(Q)}$-normalized in~$\Ee$.

\begin{Definition}
A $T$-subsystem $\Ee$ of $\Ff$ is {\it saturated\/} if
 every $\Ee$-isomorphism class contains a subgroup that is fully normalized in~$\Ee$.
\end{Definition}

If $\Ee$ is a saturated $T$-subsystem, then $T$ can be recovered from $\Ee$ as the inverse image of $\Aut_{\Ee}(S)$ under the conjugation map $S \to \Aut_{\Ff}(S)$, so sometimes we will not mention $T$.

\npar We use the same definitions for a $T$-subsystem of a pro-fusion system, with the proviso that each $\Hom_{\Ee}(P,Q)$ should be closed in $\Hom _{\Ff}(P,Q)$. If each $\Hom_{\Ee}(P,Q)$ is open in $\Hom _{\Ff}(P,Q)$ we say that $\Ee$ is open.

The closure condition implies that if $\Ff = \ilim_{i\in I} \Ff_i$, then $\Ee = \ilim_{i\in I} F_i(\Ee)$ as sets and hence that $\Ee = \ilim_{i\in I} \langle F_i(\Ee) \rangle$, where $\langle F_i(\Ee) \rangle$ denotes the $f_i(T)$-subsystem of $\Ff_i$ generated by $F_i(\Ee)$.

\begin{Example}
Let $G$ be a profinite group with Sylow pro-$p$ subgroup $S$. Let $H\lhd G$ and set $T=H \cap S$. Let $\Ee_S(H)$ be the subcategory of $\Ff_S(G)$ with objects the subgroups of $S$ and morphisms all the group homomorphisms induced by conjugation by an element of $H$. Then $\Ee_S(H)$ is a saturated $T$-subsystem of $\Ff_S(G)$.
\end{Example}

\begin{Lemma}
\label{la:equiv_fnorm2}

If $Q$ is fully $K$-normalized in $\Ee$ and $\varphi \in \Iso _{\Ee}(R,Q)$, then there is morphism $\psi : N_T^{{}^{\varphi^{-1}}K}(R)R \rightarrow N_T^K(Q)Q$ in $\Ee$ and an $\chi \in K$ such that $\psi |_R = \chi \varphi$.
\end{Lemma}

\begin{proof}
This is strictly analogous to that of Lemma~\ref{la:equiv_fnorm} and is left to the reader.
\end{proof}

\begin{Definition}
Let $\Ff$ be a pro-fusion system on $S$, $T$ a strongly $\Ff$-closed subgroup of $S$, $\Ee$ a saturated $T$-subsystem of $\Ff$ and $Q$ a subgroup of $S$. We say that $Q$ is {\it $\Ee$-radical\/} if $\Aut_{T\cap Q}(Q)=O_p(\Aut_\Ee(Q))$ and that it is {\it $\Ee$-centric\/} if $C_T(Q')\le Z(Q')$ for all $Q'$ in the $\Ee$-isomorphism class of $Q$. A subgroup $Q \leq S$ is {\it $\Ee$-essential\/} if $T\le Q$ or if $Q$ is $\Ee$-centric and $S_p\left(\,\Aut_\Ee(Q)/\Aut_{T\cap Q}(Q)\,\right)$ is disconnected (or empty), where, for a group $G$, $S_p(G)$ denotes the partially ordered set of non-trivial $p$-subgroups of $G$.
\end{Definition}

Notice that an $\Ee$-essential subgroup of $S$ is also $\Ee$-radical.

\begin{Theorem}\label{openAFT}
Let $\Ff$ be a pro-fusion system on $S$, $T$ a strongly $\Ff$-closed subgroup of $S$ and let $\Ee$ be a saturated $T$-subsystem of $\Ff$. Suppose that $P$ and $P'$ are open subgroups of $S$ and that $\varphi\in\Iso_{\Ee}(P,P')$. Then there exist open subgroups of $S$, $P=P_0,P_1,\dots,P_n=P'$ and $Q_1,Q_2,\dots,Q_n$ and morphisms $\varphi_i\in\Aut_\Ee(Q_i)$ for $1 \leq i \leq n$ such that:
\begin{enumerate}[(a)]
\item the $Q_i$ are $\Ee$-essential subgroups of $S$ and are fully normalized in~$\Ee$;
\item for each $i$, $P_{i-1}, P_i\le Q_i$ and $\varphi_i(P_{i-1})=P_i$;
\item $\varphi(u)=\varphi_n\dots\varphi_2\varphi_1(u)$ for all $u\in P$.
\end{enumerate}
\end{Theorem}

Notice that if $\Ee=\Ff$ and $S$ is finite then this is just the usual statement of Alperin's Fusion Theorem in the context of saturated fusion systems (see~\cite{puig, BLO2}). Even in the case of a finite group $G$ with normal subgroup $H$ and $\Ff=\Ff_S(G)$ and $\Ee=\Ee_S(H)$, it appears to be new.

\begin{proof}
The proof is modeled on the one for saturated fusion systems. We can use induction on the index $|S:P|$, because $P$ is open, so this index is finite. If $|S:P|=1$, then $n=1$, $S=Q_1=P_0=P_1$ and $\varphi=\varphi_1\in\Aut_\Ee(S)$. If $|S:P|>1$, let $\psi\in\Hom_\Ee(P,S)$ be such that $\psi(P)$ is fully normalized in~$\Ee$. Then $\varphi=(\psi \varphi ^{-1})^{-1}\circ\psi$, where both $\psi\in\Hom_\Ee(P,\psi(P))$ and $\psi\varphi^{-1}\in\Hom_\Ee(\varphi(P),\psi(P))$ have as image the subgroup $\psi(P)$, which is fully normalized in~$\Ee$ and has the same index in $S$ as $P$, by Lemma~\ref{la:open}.
Thus we only need to decompose morphisms in $\Ee$ with image a subgroup of $S$ that is fully normalized in~$\Ee$, because we can splice together a chain for $\psi$ with the inverse of a chain for $\psi \varphi ^{-1}$.

So now we may suppose that $P'$ is fully normalized in $\Ee$. By Lemma~\ref{la:equiv_fnorm2}, there is a morphism $\psi \in \Hom_\Ee(P\N TP,S)$ and a $\chi \in \Aut _\Ee(P')$ such that $\psi | _{P} = \chi \varphi$. If $N_T(P)$ is not contained in $P$ then $N_T(P)P$ has strictly smaller index than $P$, so we can find a chain for $\psi$ by induction.
 If $\N TP\le P$, then  $T \leq P$, as $P$ is open, and so $P$ is $\Ee$-essential. From the last condition in the definition of a subsystem, we know that $\varphi(u)u^{-1}\le T$ for all $u\in P$, which yields $P=P'$ and $\varphi \in\Aut_\Ee(P)$. Therefore we are reduced to the case of decomposing an automorphism of a subgroup that is fully normalized in $\Ee$.

So let $P$ be an open  subgroup of $S$ that is fully normalized in~$\Ee$ and $\chi\in\Aut_\Ee(P)$. There exists a $\tilde\chi\in\Hom_\Ee(N^T_{\chi},T)$ extending $\chi$. If $P$ is not $\Ee$-centric then there is a subgroup $P'$ and a $\theta \in \Iso_{\Ee}(P',P)$ such that $C_T(P') \not \leq P'$. But $\theta$ extends to $\tilde \theta : C_T(P')P' \rightarrow C_T(P)P$, and so $C_T(P) \not \leq P$. Since $C_T(P)P \leq N^T_\chi$, this implies that $N^T_\chi$ is strictly bigger than $P$, and so $\tilde\chi$ can be decomposed by induction. Thus we may assume that $P$ is $\Ee$-centric. If $P$ happens to be $\Ee$-essential, then $\chi$ already has the desired form. Otherwise, $\Ss_p\left(\,\Aut_\Ee(P)/\Aut_{T\cap P}(P)\,\right)$ is connected, and so $\Ss_p(\Aut_\Ee(P))_{>\Aut_{T\cap P}(P)}$ is also connected. Thus there exist pro-$p$ subgroups $U_1,\ldots, U_{m-1}$ and $R_1,\ldots,R_m$ of $\Aut_\Ee(P)$, all  strictly containing $\Aut_{P\cap T}(P)$, such that $U\ge R_1\le U_1\ge R_2\le U_2\ge\ldots\le U_{m-1}\ge R_m\le\,^{\chi}U$. We may suppose that the $U_i$ are Sylow pro-$p$ subgroups of $\Aut_\Ee(P)$. Because $U:= \Aut_T(P)$ is a Sylow $p$-subgroup of $\Aut_\Ee(P)$, there exist $\chi_i\in\Aut_\Ee(P)$ such that $U_i=\,^{\chi_i}U$ for each $1\le i\le m-1$. Thus $U$ and $\,^{\chi}U$ are connected by the path
$$U\ge R_1\le\,^{\chi_1}U\ge R_2\le\ldots\,^{\chi_{m-1}}U\ge R_m\le\,^{\chi}U\,.$$
Setting $\chi_0:=\id$ and $\chi_m:=\chi$, we have $\,^{\chi_{i-1}}R_i\le U\ge\,^{\chi_i}R_i$ for all $1\le i\le m$. Let $\theta_i:=\chi_i\chi_{i-1}^{-1}\in\Aut_\Ee(P)$; then $^{\theta_i}(\,^{\chi_{i-1}}R_i)=\,^{\chi_i}R_i\le U$, and hence  the image of $N^T_{\theta_i}$ in $\Aut_\Ee(P)$ contains $\,^{\chi_{i-1}}R_i$. The latter strictly contains $\Aut_{P \cap T}(P)$, so $N^T_{\theta_i}$ strictly contains~$P$, hence has smaller index. Because $P$ is fully normalized in~$\Ee$, $\theta_i$ extends to $\tilde{\theta_i}\in\Hom_\Ee(N^T_{\theta_i},S)$, which can be expressed as a chain by induction.
 This is true for each $1\le i\le m$, so we obtain a chain for $\chi=\theta_m\ldots\theta_1$, as required.
\end{proof}

The next lemma is a generalization to $T$-subsystems of Theorem~\ref{thmPuig0}. We follow here the proof presented in~\cite[proof of Proposition~5.11]{craven}.

\begin{Lemma}\label{Lemma28Finite}
Let $\Ff$ be a fusion system on a finite $p$-group $S$, let $\Ee$ be a saturated $T$-subsystem of $\Ff$ and let $N$ be a strongly $\Ff$-closed subgroup of $T$. Then for any $\varphi\in\Hom_\Ee(P,Q)$ there exists $\tilde\varphi\in\Hom_\Ee(PN, QN)$ such that $\varphi$ and $\tilde\varphi$ induce the same map from $PN/N$ to $QN/N$.
\end{Lemma}

\begin{proof}
We use induction on $|S:N|$. By the relative version of Alperin's Fusion Theorem, it is enough to prove the result for $\varphi\in\Aut_\Ee(P)$ when $P$ is $\Ee$-essential and fully normalized in~$\Ee$. Let $K:=\Ker(\Aut_\Ee(P)\to\Aut(PN/N))$; $\Aut^K_T(P)$ is thus a Sylow $p$-subgroup of $K$. Let $A:=\Aut_\Ee(P)$;  the Frattini Argument applied to $A$ and $K$ yields $\Aut_\Ee(P)=KN_A(\Aut^K_T(P))$, and so $\varphi=\chi\psi$ with $\chi\in K$ and $\psi\in N_A(\Aut^K_T(P))$. Hence $\varphi$ and $\psi$ induce the same morphism in $\Aut(PN/N)$. Since $P$ is fully normalized in~$\Ee$, $\psi$ extends to~$N^T_\psi$. We are done by induction if $N^T_\psi$ is strictly bigger than $P$. Suppose that $P=N^T_\psi$. Because $\psi\in N_A(\Aut^K_T(P))$, it is clear that $N^T_\psi$ contains $N^K_T(P)$, and hence it contains $N_{N}(P)$, since $N \leq T$. It follows that  $N_{N}(P)$ is a subgroup of $P$; as we are dealing with finite $p$-groups, this implies that $N \leq P$. But in this case the result is clear.
\end{proof}

\begin{Definition}
A $T$-subsystem $\Ee$ of a pro-fusion system $\Ff$ is {\it pro-saturated} if it is the inverse limit of saturated subsystems on finite groups, $\Ee_i \leq \Ff_i$, for some expression $\Ff = \ilim \Ff_i$ (not necessarily all).
\end{Definition}

The next result is a generalization of Lemma~\ref{Lemma28Finite} to pro-fusion systems.

\begin{Lemma}\label{Lemma28Infinite}
Let $\Ff$ be a pro-fusion system on a pro-$p$ group $S$, let $\Ee$ be a pro-saturated $T$-subsystem of $\Ff$ and let $N$ be an open strongly $\Ff$-closed subgroup of $T$. Given any $\varphi\in\Hom_\Ee(P,Q)$ we can find $\tilde\varphi\in\Hom_\Ee(PN,QN)$ such that $\varphi$ and $\tilde\varphi$ induce the same morphism from $PN/N$ to $QN/N$.
\end{Lemma}

\begin{proof}
Let $\Ff = \ilim \Ff_i$ and  $\Ee = \ilim \Ee_i$, with $\Ee_i \leq \Ff_i$. The result for some $\Ee_i$ can be obtained by considering the image $\oline\varphi_i:PN_i/N_i\to QN_i/N_i$ of $\varphi$ in $\Ee_i$. By Lemma \ref{Lemma28Finite}, there exists a $\tilde\varphi_i:PNN_i/N_i\to QNN_i/N_i$ that has the same image as $\varphi$ modulo $NN_i/N_i$. Let $X_i$ be the set of all possible such $\tilde\varphi_i$. The $X_i$ are non-empty, and they form an inverse set with respect to the functors $F_{ij}$, so $\ilim_{i\in I} X_i$ is non-empty. Take $\tilde\varphi\in\ilim_{i\in I} X_i\subseteq \Hom_\Ee(PN,QN)$.
\end{proof}

\begin{Lemma}
\label{la:opensystem}
Let $\Ff$ be a pro-fusion system on $S$ and $\Ee$ an open $T$-subsystem of $\Ff$. Then for any $\varphi\in\Hom_\Ee(P,S)$ there exists an open strongly $\Ff$-closed subgroup $N$ of $S$ with the property that if $\theta\in\Hom_\Ff(P,S)$ is such that $\theta \cong \varphi$ modulo $N$ then $\theta\in\Hom_\Ee(P,S)$.
\end{Lemma}

\begin{proof}
Given that $\Hom _\Ee (P,S)$ is an open subset of $\Hom _\Ff (P,S)$, there is an $i\in I$ such that $(F_i | _{\Hom _\Ff (P,S)})^{-1} F_i(\varphi) \subseteq \Hom _{\Ee}(P,S)$. Let $N=N_i$.
\end{proof}

\npar\label{CountableSubsystems} We say that $\Ff$ has a \emph{countable convergent sequence of open pro-saturated subsystems} if there exist open subsystems $\Ff = \Ee^1 \geq \Ee^2 \geq \cdots $ such that $\Ee^i$ is a pro-saturated $T^i$-subsystem and $S = \ilim _{i \in \mathbb N} S/T^i$. By the last condition in Definition \ref{subsystem} it follows that $\cap_i \Ee^i$ is the set of inclusions between the subgroups of $S$.

\begin{Example} If $\Ff= \Ff_S(G)$ is the fusion system associated to a countably based profinite group $G$ and $\{ H_i \} _{i \in \mathbb N}$ is a sequence of open subgroups such that $G=\ilim_{i\in \mathbb N} G/H_i$,  then the $\Ee_i :=\Ee_S(H_i)$  have the above property.
\end{Example}

\begin{Lemma}\label{Lemma27}
Let $\Ff$ be a saturated fusion system on $S$ with a countable convergent sequence of open pro-saturated $T^i$-subsystems $\Ee^i$. Then, for any open strongly $\Ff$-closed subgroup $N$ of $S$ and any $P \leq S$, there exists an $i$ such that if $\theta\in\Hom_{\Ee^i}(P,S)$ then $\theta  \simeq\id _P$ modulo $N$.
\end{Lemma}

\begin{proof}
Consider the image of $\Hom_{\Ee^i}(P,S)$ in $\Hom_{\Ff/N}(PN/N, S/N)$. Since $\bigcap_{i\ge 1}\Hom_{\Ee^i}(P,S)=\{\inc_P\}$ and $\Hom_{\Ff/N}(PN/N,S/N)$ is finite, we must have, by compactness, that the image of $\Hom_{\Ee^i}(P,S)$ is $\{\inc_{PN/N}\}$ for large enough $i$.
\end{proof}

\begin{Theorem}\label{ProAFTgene}
Let $\Ff$ be a pro-fusion system on a pro-$p$ group $S$ that has a countable convergent sequence of open pro-saturated $T^i$-subsystems $\Ee^i$ and let $P$, $P'$ be two closed subgroups of $S$. Then for each $\varphi\in\Iso_\Ff(P,P')$ there exist subgroups $P=P_0,P_1,P_2,\dots$ and open subgroups $Q_1,Q_2,Q_3,\dots$ of $S$ and morphisms $\varphi_i\in\Aut_\Ff(Q_i),\,i\ge 1$ such that:
\begin{enumerate}[(a)]
\item for each $j$ there exists an $n_j$ such that for any $i \geq n_j$, $\varphi_i\in\Aut_{\Ee^j}(Q_i)$ and $Q_i$ is $\Ee^j$-essential and fully normalized in~$\Ee^j$;
\item $P_{i-1},\,P_i\le Q_i$, $\varphi_i(P_{i-1})=P_i$;
\item $\varphi (u) = \dots\varphi_i\varphi_{i-1}\dots\varphi_2\varphi_1 (u)$for all $u \in P$ and the infinite composition converges.
\end{enumerate}
\end{Theorem}
\begin{proof}
We construct the Alperin chain by successive approximation. Recall that $\Ee^1=\Ff$, and start with $R_0=P$ and $\theta_0=\varphi$. Given $\theta_i\in\Iso_{\Ee^i}(R_i,P')$, we show that there is a morphism $\psi_{i+1}\in\Iso_{\Ee^i}(R_i,R_{i+1})$, given by an Alperin chain in $\Ee^i$ as in the theorem, and a morphism $\theta_{i+1}\in\Iso_{\Ee^{i+1}}(R_{i+1},P')$ such that $\theta_i=\theta_{i+1}\psi_{i+1}$. Thus $\varphi = \theta_i \psi_i \cdots \psi_1$. We consider the infinite chain $\cdots\psi_3\psi_2\psi_1$ and replace each $\psi_i$ by its Alperin chain to obtain the Alperin chain required.

The first two conditions are satisfied, with $n_j$ set equal to the sum of the lengths of the chains for $\psi_i, \ldots , \psi_1$. Given an open strongly $\Ff$-closed subgroup $N$ of $S$, choose $j$ as in Lemma~\ref{Lemma27}. If $i \geq j$ then $\varphi _i \cdots \varphi_1= \varphi _i \cdots \varphi_{N_j+1} \psi_j \cdots \psi_1=\varphi _i \cdots \varphi_{N_j+1} \theta_j^{-1} \varphi$. But $\varphi _i \cdots \varphi_{N_j+1} \theta_j^{-1} \in \Ee_j$, and so $\varphi _i \cdots \varphi_{N_j+1} \theta_j^{-1} \simeq \id_{P'}$ modulo $N$. Thus $\varphi \simeq \varphi _i \cdots \varphi_1$ modulo $N$, which shows that the infinite composition converges.

Now for the construction.
Let $\inc_{P'}$ denote the inclusion of $P'$ in $S$; Lemma~\ref{la:opensystem} allows us to choose an open strongly $\Ff$-closed subgroup $N$ of $S$ such that if $\chi \simeq \inc_{P'}$ modulo $N$ then $\chi\in\Hom_{\Ee^{i+1}}(P,S)$. Replacing $N$ by $N \cap T^{i+1}$ if necessary, we may assume that $N \leq T^{i+1}$.
Given $\theta_i\in\Hom_{\Ee^i}(R_i,P')$,
apply Lemma~\ref{Lemma28Infinite} to $\theta_i$ to obtain a $\tilde\theta_i\in\Hom_{\Ee^i}(R_iN,P'N)$ such that $\theta_i$ and $\tilde\theta_i$ are equal modulo $N$. Set $\psi _{i+1}:= \tilde\theta_i|_{R_i}$, $R_{i+1}:=\psi_{i+1}(R_i)$ and $\theta_{i+1}:=\theta_i \psi_{i+1}^{-1}$. Since $R_iN$ is an open subgroup of $S$, we can apply Theorem~\ref{openAFT} relative to $\Ee^i$ to obtain an Alperin chain in $\Ee_i$ for $\tilde\theta_i$ and thus also for $\psi_{i+1}$. Moreover, $\theta_{i+1}^{-1}=\inc_{P'}$ modulo $N$; hence $\theta_{i+1}^{-1}\in\Iso_{\Ee^{i+1}}(P',R_{i+1})$, and so $\theta_{i+1}\in\Iso_{\Ee^{i+1}}(R_{i+1},P')$
\end{proof}

\begin{Lemma}
\label{la:finess}
Suppose that for some $\varphi \in \Iso _\Ff (P,P')$ there is a finite Alperin chain in $\Ff$ as in the conclusion of the preceding theorem, although it is not required to satisfy condition (a). Then there is such a chain with all of the $Q_i$  $\Ff$-essential and fully normalized in $\Ff$.
\end{Lemma}

\begin{proof}
 Since the $Q_i$ are open subgroups, we can apply Theorem~\ref{openAFT} to each $\varphi_i\in\Aut_\Ff(Q_i)$ in order to decompose it as a finite composition of restrictions of automorphisms of $\Ff$-essential $\Ff$-normalized subgroups. Now splice these chains together.
\end{proof}

\begin{Remark}
The proof of the preceding lemma will not work if the original chain is infinite, because of problems with convergence.
\end{Remark}

\section{The length of an Alperin chain}
It really is necessary to allow the possibility of an infinite chain in the statement of Theorem~\ref{ProAFTgene}, as we will now show with an example.

Let $G$ be a finite group with Sylow $p$-subgroup $S$. Let $P,P' \leq S$ be two subgroups that are conjugate by an element $g \in G$ but are not conjugate by any element of $N_G(S)$. For example, if $G = \Gl _3 (\mathbb F_p)$ and $S$ consists of matrices of the form $ \left( \begin{smallmatrix} 1&*&*\\ 0&1&* \\ 0&0&1 \end{smallmatrix} \right)$, $P$ those of the form $\left( \begin{smallmatrix} 1&*&0\\ 0&1&0 \\ 0&0&1 \end{smallmatrix} \right)$, $P'$ those of the form $\left( \begin{smallmatrix} 1&0&0\\ 0&1&* \\ 0&0&1 \end{smallmatrix} \right)$ and $g=\left( \begin{smallmatrix} 0&0&1\\ 1&0&0 \\ 0&1&0 \end{smallmatrix} \right)$.

Now consider the cartesian product of infinitely many copies of $G$, indexed by some set $I$. We will denote this by $\prod G_i$. The subgroups $\prod P_i$ and $\prod P'_i$ are conjugate by $\prod g_i$.

An open subgroup of $\prod S_i$ is of the form $X = Y \times \prod _{i \not \in J} S_i$, for some finite set $J \subset I$. Thus an automorphism of $X$ will have as $i$-coordinate an element of $N_G(S_i)$ for all but finitely many $i$. The same will be true of a composition of finitely many such automorphisms. In particular, such a composition cannot be equal to $\prod g_i$.

The rest of this section contains various facts about the length of an Alperin chain.

Let $\Ff$ be a pro-fusion system on a pro-$p$ group $S$. Let $P, P' \leq S$ and $\varphi \in \Iso _{\Ff}(P,P')$. By an Alperin chain in $\Ff$ for $\varphi$ we mean a sequence of subgroups of $S$, $P=P_0, P_1, \ldots ,P_n=P'$, a sequence of open subgroups $Q_1, \ldots , Q_n$ and morphisms $\varphi_i \in \Aut_{\Ff}(Q_i)$ such that $P_{i-1},P_i\leq Q_i$, $\varphi_i(P_{i-1})=P_i$ and $\varphi = \varphi _n \cdots \varphi_1$. By a chain of essential subgroups we mean such a chain with the extra condition that all the $Q_i$ are $\Ff$-essential.

The length of the chain is defined to be the number of subgroups $Q_i$ that are \emph{not} equal to $S$. It is known that a chain can be rearranged so that all the $Q_i$ that are equal to $S$ appear at the end and all their automorphisms can thus be merged into one (see, for example,~\cite[2.8]{stancu}). Thus not counting the $Q_i$ equal to $S$ will only change the length of a shortest chain by at most one.

\begin{Definition}
The Alperin length of $\varphi$ is the minimum length of a chain for $\varphi$ as defined above; we denote it by $\Alp _{\Ff} ( \varphi )$. If no such (finite) chain exists, we write $\Alp _{\Ff} ( \varphi ) = \infty$. If we require a chain of essential subgroups we write $\Alp ^{\ess}_{\Ff} (\varphi)$. If we relax the condition that the $Q_i$ be open then we write $\Alp^{\clo}_{\Ff}(\varphi)$.
\end{Definition}

\begin{Lemma}\label{infchain}
We have $\Alp ^{\ess}_{\Ff} (\varphi) \geq \Alp _\Ff (\varphi) \geq \Alp _{\Ff}^{\clo} ( \varphi )$ and $\Alp ^{\ess}_{\Ff} (\varphi) \geq  \Alp^{\clo,\ess}_\Ff(\varphi) \geq \Alp _{\Ff}^{\clo} ( \varphi )$.
If $\Alp _{\Ff} (\varphi) < \infty$ then $\Alp ^{\ess}_{\Ff} (\varphi) < \infty$.
\end{Lemma}

\begin{proof}
The first statement is clear from the definitions and the second is the content of Lemma~\ref{la:finess}.
\end{proof}

For simplicity we will mostly, in what follows, deal only with pro-fusion systems associated to profinite groups. In this case a fixed choice of Sylow pro-$p$ subgroup $S$ is implicit and we write $\Alp _G (\varphi)$ for $\Alp _{\Ff_S(G)}(\varphi)$.

\begin{Lemma}\label{lengthproduct}
If we have a collection of profinite groups $\{ G_i | i \in I \}$, possibly infinite, with Sylow pro-$p$ subgroups $S_i$, pro-$p$ subgroups  $P_i,P'_i \leq S_i$ and morphisms $\varphi_i \in \Iso _{G_i}(P_i,P_i')$ and form the cartesian product, then $\sum _i \Alp _{G_i} (\varphi_i) \geq \Alp_{\prod G_i}(\prod \varphi _i) \geq \sup \{ \Alp _{G_i}(\varphi _i)\}$ and $\Alp^{\clo}_{\prod G_i}(\prod \varphi _i) = \sup_i \{ \Alp _{G_i}(\varphi _i)\}$.
\end{Lemma}

\begin{proof}
An Alperin chain in $\prod G_i$ for $\prod \varphi _i$ can be projected to a chain in any $G_i$ for $\varphi_i$; this gives the inequality with $\sup$ in either case. For the other inequality, we may assume that $\sup_i \{ \Alp _{G_i}(\varphi _i)\}$ (or its closed version) is finite, equal to $N$, say.
For each $i$, let $\varphi _i = \varphi _i ^{n_i} \cdots \varphi _i ^1$ be an Alperin chain of minimal length, with $\varphi_i^j \in \Aut (Q_i^j)$, and with the $Q_i^j$ open subgroups if that is the case we are treating.

 In the closed case, extend each chain to have length $N$ by adding copies of the identity morphism on $S_j$. Then the $\prod_i \varphi_i^j \in \Aut_{\prod_i G_i}(\prod_i Q_i^j)$ form an Alperin chain of length $N$.

In the open case we may assume that $\sum _i \Alp _{G_i} (\varphi_i) < \infty$. Thus $\Alp _{G_i} (\varphi_i)$ is only non-zero for finitely many $i$, say for $i \in K$, and $\varphi _i \in \Aut_{G_i}(S_i)$ for $i \not \in K$. For $i \in K$, extend each $\varphi_i^j$  to $\varphi _i^j \times 1 \in \Aut_{\prod_i G_i}(Q_i^j \times \prod _{j \ne i} S_j)$. These form an Alperin chain of length $\sum _i \Alp _{G_i} (\varphi_i)$, for which the composition is equal to $\prod_{i \in K} \varphi_i\times \prod_{i \not \in K}1$. We can obtain $\prod_i \varphi_i$ as composition by adding the term $\prod _{i \in K} 1 \times \prod_{i \not \in K} \varphi_i \in \Aut_{\prod_i G_i}(\prod_i S_i)$, which does not count towards the length.
\end{proof}

\begin{Proposition}\label{essentialchain}
If we have a collection of profinite groups $G_i$ with Sylow pro-$p$ subgroups $S_i$, pro-$p$ subgroups  $P_i,P'_i \leq G_i$ and morphisms $\varphi_i \in \Iso _{G_i}(P_i,P_i')$ and form the cartesian product, then $\Alp^{\ess}_{\prod G_i}(\prod \varphi _i) = \sum  \Alp ^{\ess} _{G_i}(\varphi _i)$.
\end{Proposition}

\begin{proof}
The inequality $\sum_i  \Alp ^{\ess} _{G_i}(\varphi _i) \geq \Alp^{\ess}_{\prod G_i}(\prod \varphi _i)$ is proved just as in the open case of the preceding lemma.

For the reverse inequality, suppose that we have a chain for $\prod_i \varphi _i$ in terms of $\theta ^j \in \Aut_{\prod_i G_i}(Q^j)$, $1 \leq j \leq N$. By Corollary~\ref{cor:essential} below, $Q^j$ is of the form $E_{i(j)} \times \prod _{i \ne i(j)} S_i$ for some $i(j)$ and some open essential subgroup $E_{i(j)}$ of $S_{i(j)}$. Thus $\theta^j = \prod_i \theta^j_i$ with $\theta^j_i \in \Aut_{G_i}(S_i)$ if $i \ne i(j)$ and $\theta^j_i \in \Aut_{G_i}(E_i)$ if $i = i(j)$.

For fixed $i$, the $\theta_i^j$ form an Alperin chain for $\varphi_i$. By our convention that automorphisms of the Sylow pro-$p$ subgroup do not count towards the length, each $j$ only counts towards the length of the chain for one $i$. The inequality follows.
\end{proof}

We say that $p$ divides the order of a profinite group if a Sylow pro-$p$ subgroup is non-trivial.

\begin{Lemma}\label{below1}

Let $A$ and $B$ be profinite groups and let $P$ be a pro-$p$ subgroup of $A \times B$ that is either finite or is open in a Sylow pro-$p$ subgroup of $A \times B$.
\begin{enumerate}
\item
If $P$ is radical then $P=Q \times R$, where $Q$ is radical in $A$ and $R$ is radical in $B$.
\item
If $P$ is essential then $P=Q \times R$, where $Q$ is essential in $A$ and $R$ is essential in $B$. Furthermore, $p$ does not divide both $|N_A(Q)/Q|$ and $|N_B(R)/R|$.
\item
If $P$ is open in a Sylow pro-$p$ subgroup of $A \times B$ and is also essential, then $P=Q \times R$, where either $Q$ is essential in $A$ and $R$ is a Sylow pro-$p$ subgroup of $B$, or $R$ is essential in $B$ and $Q$ is a Sylow pro-$p$ subgroup of $A$.
\end{enumerate}
\end{Lemma}

\begin{proof}
Denote the projections by $\pi_A$ and $\pi_B$. We have $P \leq \pi_A P \times \pi_B P $; the inclusion is open, by the hypotheses on $P$.

If $P=\pi_A P \times \pi_B P$ then it is easy to verify that $\pi_A P$ is radical in $A$ and $\pi_B P$ is radical in $B$.

If $P \ne \pi_A P \times \pi_B P$ then $P \ne N_{\pi_A P \times \pi_B P}(P) \leq O_p(N_{A \times B}(P))$, contradicting the assumption that $P$ is radical.

For the second part, we know from the first part that $P=Q \times R$. Since $Q$ is radical, it is also centric in $A$; thus $p$ divides $|N_A(Q)/Q|$ if and only if it divides $|N_A(Q)/C_A(Q)Q|$. The same applies to $R$. The result follows from the next lemma applied to $N_{A \times B}(Q \times R)/C_{A \times B}(Q \times R)(Q \times R) \cong N_A(Q)/C_A(Q)Q \times N_B(R)/C_B(R)R$.

For the third part, suppose that $p$ does not divide $|N_A(Q)/Q|$ and let $S$ be a Sylow pro-$p$ subgroup of $A$ that contains $Q$. Then $p$ does not divide $|N_S(Q)/Q|$, so $N_S(Q) \leq Q$ and thus $S \leq Q$, since $Q$ is open in $S$.
\end{proof}

\begin{Lemma}
The poset $S_p(A \times B)$ is connected whenever both $|A|$ and $|B|$ are divisible by $p$. If $|A|$, say, is not divisible by $p$, then $S_p(A \times B)$ is connected if and only if $S_p(B)$ is connected.
\end{Lemma}

\begin{proof}
If  both $|A|$ and $|B|$ are divisible by $p$, choose non-trivial pro-$p$ subgroups $X \leq A$ and $Y \leq B$. We will show that any non-trivial pro-$p$ subgroup $P \leq A \times B$ is connected to $X \times 1$.

If $\pi_B P \ne 1$ then we can use the chain $P \leq \pi_A P \times \pi_B P \geq 1 \times \pi _B P \leq X \times \pi _B P \geq X \times 1$.

If $\pi _A P =1$ then $P = \pi _A P \times 1 \leq \pi _A P \times Y \geq 1 \times Y \leq X \times Y \geq X \times 1$.

If $|A|$ is not divisible by $p$ then $S_p(A \times B) \equiv S_p(B)$, and the second part follows.
\end{proof}

\begin{Corollary}
\label{cor:essential}
An open essential subgroup of $\prod_i S_i \leq \prod _i G_i$ is of the form $E_j \times \prod _{i \ne j}S_i$ for some $j$ and an open essential subgroup $E_j$ of $S_j$.
\end{Corollary}

\begin{proof}
Let $P$ be the open essential subgroup. For each $j$, apply Lemma~\ref{below1} to $G_j \times \prod_{i \ne j}G_i$ to see that the projection of $P$ to its $j$-coordinate is both a factor of $P$ and essential. Thus $P= \prod _i P_i$, with the $P_i$ essential. If $P_k$, say, is not a Sylow pro-$p$ subgroup of $S_k$ and $j \ne k$, then apply Lemma~\ref{below1} to $(G_j \times G_k) \times \prod_{i \ne j,k} G_i$ to see that $P_j \times P_k$ is essential and hence, by the lemma again, $P_j=S_j$.
\end{proof}

\begin{Theorem}
Let $G_i$ be an infinite collection of finite groups, each with $p$-subgroups $P_i,P'_i \leq G_i$ and $\varphi_i \in \Iso_{G_i}(P_i,P'_i)$, and suppose that $\Alp_{G_i}(\varphi_i) \geq 1$. Then $\Alp_{\prod G_i}(\prod \varphi_i) = \Alp^{\ess}_{\prod G_i}(\prod \varphi_i) = \infty$.
\end{Theorem}

\begin{proof} We must have $\Alp^{\ess}_{G_i}(\varphi_i) \geq 1$, so $\Alp^{\ess}_{\prod G_i}(\prod \varphi_i) \geq \sum 1 = \infty$, by Proposition~\ref{essentialchain}. The same must be true for $\Alp_{\prod G_i}(\prod \varphi_i)$, by Proposition~\ref{infchain}.
\end{proof}

This provides another proof of the fact that it is necessary to allow infinite chains in Theorem~\ref{ProAFTgene}.
A direct approach to showing that infinite chains are necessary, even for $\Alp^{\clo}$, would be to find a sequence of finite groups $G_i$ and $P,P' \leq S_i$, $\varphi_i \in \Iso(P,P')$ such that $\Alp_{G_i}(\varphi_i) \rightarrow \infty$ as $i \rightarrow \infty$, and then apply Lemma~\ref{lengthproduct} to their product. It seems to be generally believed that such a sequence exists; however, it seems difficult to find a lower bound for $\Alp_G(\varphi)$ in examples.

Notice that if such a sequence does not exist at the prime $p$ then there exists an integer $n$ such that, at the prime $p$, $\Alp_G(\varphi) \leq n$ for all finite groups $G$ and all morphisms $\varphi$. In this case we will say that $\Alp$ is uniformly bounded at $p$.

Sometimes it is convenient to consider all the morphisms in a profusion system $\Ff$ as a set, $\Mor(\Ff)$ (morphisms with different domains or codomains are considered to be different).

\begin{Lemma}
If $\Ff \cong \ilim _i \Ff_i$ then $\Mor(\Ff) \cong \ilim _i \Mor(\Ff_i)$.
\end{Lemma}

\begin{proof}
There is clearly a natural injective map $\Mor(\Ff) \to \ilim _i \Mor(\Ff_i)$. An element of the codomain is a compatible collection of morphisms $\varphi : P_i \to Q_i$; these combine to define a morphism $\varphi : \ilim _i P_i \to \ilim _i Q_i$.
\end{proof}

We will now regard $\Mor(\Ff)$ as a profinite set in this way.

The ordered set of the morphisms in an Alperin chain for $\varphi$ of length at most $n$ can be considered as an element of $\Mor(\Ff)^n$ (shorter chains can be extended by the identity map). This element completely determines the Alperin chain.

\begin{Proposition}
\label{limitlength}
If $\Ff=\ilim \Ff_i$, then
$\Alp^{\clo}_\Ff(\varphi) \leq \sup \{ \Alp _{\Ff_i}(F_i(\varphi)) \}$.
\end{Proposition}

\begin{proof}
If $\sup \{ \Alp _{\Ff_i}(F_i(\varphi_i))\}= \infty$ there is nothing to prove, so assume that $\sup \{ \Alp _{Ff_i}(F_i(\varphi _i))\} =n$.

Let $C_i$ be the set of all Alperin chains in $\Ff_i$ for $F_i(\varphi)$ of length $n$, considered as a subset of $\Mor(\Ff_i)^n$. The sets $C_i$ are non-empty, by hypothesis, and if $j \geq i$ there is a natural map $C_j \to C_i$. Thus we can form $\ilim_i C_i$ and it is non-empty; an element of it yields an Alperin chain in closed subgroups in $\Ff$ of length $n$ for $\varphi$.
\end{proof}

\begin{Corollary}
It is always possible to find a finite Alperin chain in closed subgroups for any profinite group and any morphism at the prime $p$, if and only if $\Alp$ is uniformly bounded at $p$. If there is a uniform bound then it also applies to profinite groups.
\end{Corollary}

\begin{proof}
This is a consequence of Proposition~\ref{limitlength} and Lemma~\ref{lengthproduct}.
\end{proof}


\begin{thebibliography}{1}

\bibitem{aschbacher:NormalSubsystems}
M.~Aschbacher,
\emph{Normal subsystems of fusion systems},
Proc. London Math. Soc. \textbf{97} (2008), 239--271.


\bibitem{benson}
D.J.~Benson,
Representations and Cohomology II: Cohomology of Groups and Modules, \emph{Cambridge Studies in Advanced Mathematics} 31, 1991.

\bibitem{BCGLO2}
C.~Broto, N.~Castellana, J.~Grodal, R.~Levi, B.~Oliver,
\emph{Extensions of $p$-local finite groups},
Trans. Amer. Math. Soc. {\bf 359} (2007) 3791--3858.

\bibitem{BLO1}
C.~Broto, R.~Levi, B.~Oliver,
\emph{Homotopy equivalences of p-completed classifying spaces of finite groups},
Invent. Math. \textbf{152} (2003) 611--664.

\bibitem{BLO2}
C.~Broto, R.~Levi, B.~Oliver,
\emph{The homotopy theory of fusion systems},
J. Amer. Math. Soc. \textbf{16} (2003) 779--856.

\bibitem{BLO4}
C.~Broto, R.~Levi, B.~Oliver,
\emph{A geometric construction of saturated fusion systems},
Contemp. Math. \textbf{399} (2006), 11--40.

\bibitem{craven} D.A.~Craven,
\emph{Control of fusion and solubility in fusion systems},  J. Algebra  \textbf{323}  (2010) 2429--2448.

\bibitem{cravenbook} D.A.~Craven,
The Theory of Fusion Systems, \emph{Cambridge Studies in Advanced Mathematics} 131, 2011.

\bibitem{gilottiribesserena}
A.L.~Gilotti, L.~Ribes, L.~Serena, \emph{Fusion in profinite groups}, Ann. Mat. Pura Appl. (4) \textbf{177} (1999), 349-362.

\bibitem{puig}
L.~Puig  \emph{Frobenius Categories\/},  J Algebra {\bf 303} (2006) 309-357.

\bibitem{puig:book}
L.~Puig, \emph{Frobenius categories versus Brauer blocks}, Progress in Mathematics \textbf{274}, Birkh\"auser Verlag, 2009.

\bibitem{Linckelmann}
M.~Linckelmann, \emph{Simple fusion systems and the Solomon 2-local groups}, J. Algebra \textbf{296}  (2006)  385--401.

\bibitem{rz}
L.~Ribes, P.~Zalesskii,
Profinite groups,
\emph{Ergebnisse der Mathematik und ihrer Grenzgebiete} 40, Springer-Verlag,  2000.

\bibitem{RobertsShpectorov}
K.~Roberts, S.~Shpectorov,
\emph{ On the definition of saturated fusion systems\/},
J. Group Theory  \textbf{12}  (2009) 679--687.

\bibitem{serre}
J.-P.~Serre,
Cohomologie Galoisienne,
\emph{ Lecture Notes in Mathematics} 5, Springer-Verlag, 1964.

\bibitem{stancu}
R.~Stancu, \emph{Control of fusion in fusion systems}, J. Algebra Appl. \textbf{5} (2006) 817--837.

\bibitem{wilson}
J.S.~Wilson,
Profinite groups,
\emph{London Mathematical Society Monographs}, new series 19, Oxford University Press, New York, 1998.

\bibitem{zel}
E.I.~Zel'manov,
\emph{On periodic compact groups},
Israel J. Math., \textbf{77} (1992) 83--95.

\end{thebibliography}
\end{document}